\documentclass{amsart}
\usepackage{amsmath}
\usepackage{amsfonts}
\usepackage{amssymb}
\usepackage{amstext}
\usepackage{amsbsy}
\usepackage{amsopn}
\usepackage{amsthm}
\usepackage{amsxtra}
\usepackage{upref}
\usepackage{fancyhdr,amsmath, graphicx, psfrag, color ,amsfonts,layout, amssymb}
\usepackage{euscript,epsfig}

\usepackage[all,cmtip]{xy}

\newcommand{\x}{\ensuremath{\underline{x}}}
\newtheorem{theo}{Theorem}[section]
\newtheorem{coro}[theo]{Corollary}
\newtheorem{lem}[theo]{Lemma}
\newtheorem{prop}[theo]{Proposition}

\newtheorem{rema}[theo]{Remark}
\theoremstyle{definition}

\theoremstyle{remark}

\def \R{\mathbb R}
\def \N{{\mathbb N}}

\def \Z{\mathbb Z}

\def \H{\mathbb H}

\def\C{\mathbb C}
\def\F{\mathcal F}

\def\U{\mathcal U}
\def\M{\mathcal M}
\def\L{\mathcal L}
\def\x{\textbf{x}}
\def\y{\textbf{y}}
\def\w{\textbf{w}}
\def\z{\textbf{z}}

\newcommand{\SL}{{\bf\operatorname{SL}}}
\newcommand{\SO}{{\bf\operatorname{SO}}}

\begin{document}

\title[Unipotent frame flows]{On topological and measurable dynamics of unipotent frame flows for hyperbolic manifolds} 
\author{ Fran\c cois MAUCOURANT, Barbara SCHAPIRA}


\begin{abstract}
We study the dynamics of unipotent flows on frame bundles of hyperbolic manifolds of infinite volume. 
We prove that they are topologically transitive, and that the natural invariant measure, the so-called "Burger-Roblin measure", is ergodic, 
as soon as the geodesic flow admits a finite measure of maximal entropy, and this entropy is strictly
greater than the codimension of the unipotent flow inside the maximal unipotent flow. The latter result generalises a Theorem of Mohammadi and Oh.

\end{abstract}

\maketitle


\section{Introduction}

\subsection{Problem and State of the art}
 
 For $d\geq 3$, let $\Gamma$ be a Zariski-dense, discrete subgroup of $G={\SO}_o(d,1)$.
 Let $N$ be a maximal unipotent subgroup of $G$ (hence isomorphic to $\R^{d-1}$), 
and $U\subset N$ a nontrivial connected subgroup (hence isomorphic to some $\R^k$ in $\R^{d-1}$). 
The main topic of this paper is the study of the action of $U$ on the space $\Gamma \backslash G$. 
Geometrically, this is the space $\mathcal{FM}$ of orthonormal frames 
of the hyperbolic manifold $\mathcal{M}=\Gamma \backslash \H^d$, and the $N$ 
(and $U$)-action moves the frame in a parallel way on the stable horosphere defined 
by the first vector of the frame. There are a few cases where such an action is well understood,
 from both topological and ergodic point of view.
 
\subsubsection{Lattices} 
 
 If $\Gamma$ has finite covolume, then Ratner's theory provides a complete description 
of closures of $U$-orbits as well as   ergodic $U$-invariant measures.
If $\Gamma$ has infinite covolume, while it no longer provide information 
about the topology of the orbits, it still classifies finite $U$-invariant measures. 
Unfortunately, the dynamically relevant measures happen to be of infinite mass. 
In the rest of the paper, we will always think of $\Gamma$ as 
a subgroup having infinite covolume.

\subsubsection{Full horospherical group}

 If one looks at the action of the whole horospherical group $U=N$, 
a $N$-orbit projects on $T^1\mathcal{M}$ onto a leaf of the strong stable 
foliation for the geodesic flow , a well-understood object, 
at least in the case of geometrically finite manifolds. 
In particular, the results of Dal'bo \cite{MR1779902} imply that for a geometrically finite manifold, 
such a leaf is either closed, or dense in an appropriate subset of $T^1\mathcal{M}$.\\

 From the ergodic point of view, there is a natural good $N$-invariant measure, 
the so-called {\em Burger-Roblin measure}, unique with certain natural properties. 
Recall briefly its construction. 
The measure of maximal entropy of the geodesic flow on $T^1\mathcal{M}$, the {\em Bowen-Margulis-Sullivan measure}, 
when {\em finite}, induces a transverse invariant measure to the strong stable foliation. 
This transverse measure is often seen as a measure on the space of horospheres, invariant under the action of $\Gamma$.  
Integrating the Lebesgue measure along these leaves leads to a measure on $T^1\mathcal{M}$, which lifts naturally
to $\mathcal{F}\mathcal{M}$ into a $N$-invariant measure, the {\em Burger-Roblin measure}. 
 
In \cite{Ro},  Roblin extended a classical result of Bowen-Marcus \cite{MR0451307}, 
and showed that, up to scalar multiple, when the Bowen-Margulis-Sullivan measure is finite, 
it induces (up to scalar multiple) the unique invariant measure supported on this space of 
 horospheres, supported in the set of  horospheres based at conical (radial) limit points.

 In particular, if the manifold $\mathcal{M}$ is geometrically finite, 
this gives a complete classification of $\Gamma$-invariant (Radon) measures on the space of horospheres, or equivalently of transverse invariant measures to the strong stable foliation. In general, 
Roblin's result says that there is a unique (up to scaling) transverse invariant measure of full support in
the set of vectors whose geodesic orbit returns i.o. in a compact set. 
 
 It is natural to try to "lift" this classification along the 
principal bundle $\mathcal{FM}\to T^1\mathcal{M}$, since the structure group is compact. This was done by Winter \cite{Winter}, who proved that, up to scaling, 
the only $N$-invariant measure of full support in the set of frames whose $A$-orbit returns i.o. in a compact set
is the Burger-Roblin measure, i.e. the natural $M$-invariant lift of the above measure. 
On geometrically finite manifolds, this statement is simpler: the Burger-Roblin is the unique (up to scaling) $N$-invariant ergodic measure of full support.

\subsubsection{A Theorem of Mohammadi and Oh}

 However, if one considers only the action of a proper subgroup $U\subset N$, the situation changes dramatically, and much less is known, because ergodicity or conservativeness of a measure with respect to a group does not imply in any way the same properties with respect to proper subgroups. In this direction, the first result is a Theorem of Mohammadi and Oh \cite{MO}, which states that, in dimension $d=3$ (in which case $dim(U)=1$) and for convex-cocompact manifolds, the Burger-Roblin measure is ergodic and conservative for the $U$-action if and only if the critical exponent $\delta_\Gamma$ of $\Gamma$ satisfies $\delta_\Gamma>1$.

\subsubsection{Dufloux recurrence results}

In \cite{Dufloux2016, Dufloux2017}, Dufloux investigates the case of small critical exponent. 
Without any assumption on the manifold, when the Bowen-Margulis-Sullivan measure is finite 
(assumption satisfied in particular when $\Gamma$ is convex-cocompact, but not only, see \cite{Peigne}), 
he proves in \cite{Dufloux2016} that the Bowen-Margulis-Sullivan is totally $U$-dissipative when $\delta_\Gamma\le\dim N-\dim U$, 
and totally recurrent when $\delta_\Gamma>\dim N-\dim U$. 
In \cite{Dufloux2017}, when the group $\Gamma$ is convex-cocompact, he proves that when $\delta_\Gamma=\dim N-\dim U$, the Burger-Roblin measure is $U$-recurrent.

\subsubsection{Rigid acylindrical 3-manifolds}

 There is one last case where more is know on the topological properties of the $U$-action, in fact in a very strong form.
Assuming $\mathcal{M}$ is a rigid acylindrical 3-manifold, McMullen, Mohammadi and Oh recently managed in \cite{McMullen2} to classify the $U$-orbit closures, which are very rigid. Their analysis relies on their previous classification of $\SL(2,\R)$-orbits \cite{McMullen}.\\

 Unfortunately, their methods relies heavily on the particular shape of the limit set (the complement of a countable union of disks), and such a strong result is certainly false for general convex-cocompact manifolds.

\subsection{Results}

 The results that we prove here divide  in two distinct parts, a topological one, and a ergodic one. Although they are independent, the strategy of their proofs follow similar patterns, a fact we will try to emphasise.

\subsubsection{Topological properties} 

 Let $A\subset G$ be a Cartan Subgroup. Denote by $\Omega\subset \mathcal{FM}$ the non-wandering set for the geodesic flow (or equivalently, the $A$-action), and by $\mathcal{E}$ the non-wandering set for the $N$-action. For more precise definitions and description of these objects, see section \ref{section2}.\\
 
 Using a Theorem of Guivarc'h and Raugi \cite{MR2339285}, we show:

\begin{theo} \label{Atopmix} Assume that $\Gamma$ is Zariski-dense.
The action of $A$ on $\Omega$ is topologically mixing.  
\end{theo}

 This allows us to deduce:

\begin{theo} \label{topologicallytransitive} Assume that $\Gamma$ is Zariski-dense.
The action of $U$ on $\mathcal{E}$ is topologically transitive. 
\end{theo}

 Both results are new. Note that, for example in the case of a general convex-cocompact manifold with low critical exponent, the existence of a non-divergent $U$-orbit is itself non-trivial, and was previously unknown.

\subsubsection{Ergodic properties}

 We will assume that $\Gamma$ is of divergent type, and denote by $\mu$ 
the Bowen-Margulis-Sullivan measure - or more precisely, its natural lift to $\mathcal{FM}$, normalised to be a probability. We are interested in the case where $\mu$ is a finite measure. 
Denote by $\nu$ the Patterson-Sullivan measure on the limit set, and $\lambda$ the Burger-Roblin measure on $\mathcal{FM}$. More detailed description of these objects is given in section \ref{section4}. \\

 The following is a strengthening of the Theorem of Mohammadi and Oh \cite{MO}. 

\begin{theo} \label{ergodicity} Assume that $\Gamma$ is Zariski-dense.
If $\mu$ is finite and  $\delta_\Gamma+\dim(U)>d-1$, then both measures $\mu$ and $\lambda$ are $U$-ergodic.
\end{theo}
 
 The hypothesis that $\mu$ is finite is satisfied for example 
 when $\Gamma$ is geometrically finite see Sullivan \cite{Sullivan1979}. But there are many other examples, 
see \cite{Peigne}, \cite{Ancona}. 
Note that the measure $\mu$ is {\em not} $U$-invariant, or even quasi-invariant; in this case, ergodicity simply means that $U$-invariant sets have zero of full measure. 
Apart from the use of Marstrand's  projection Theorem, our proof differs significantly 
from the one of \cite{MO}, and does not use compactness arguments, 
allowing us to go beyond the convex-cocompact case. 
It is also, in our opinion, simpler. Note that the work of Dufloux \cite{Dufloux2016} uses the same
assumptions as ours. \\
 
 For the opposite direction, we prove:

\begin{theo} \label{nonergodicity} Assume that $\Gamma$ is Zariski-dense.
If $\mu$ is finite
with $\delta_\Gamma+\dim(U)<d-1$,  
then $\lambda$-almost every frame is divergent.
\end{theo}

In fact, in the convex-cocompact case,   a stronger result holds: for all vectors $v\in T^1\mathcal{M}$ and almost all frames $\x$ 
in the fiber of $v$, the orbit $\x U$ is divergent, see Theorem \ref{divergence} for details.

\subsection{Overview of the proofs}

\subsubsection{Topological transitivity}

 The proof of the topological transitivity can be summarised as follows.
 \begin{itemize}
 \item The $U$-orbit of $\Omega$ is dense in $\mathcal{E}$ (Lemma \ref{UOmegadense}).
 \item The mixing of the $A$-action (Theorem \ref{Atopmix}) implies that there are couples $(\x,\y)\in \Omega^2$, generic in the sense that their orbit by the diagonal action of $A$ by negative times on $\Omega^2$ is dense in $\Omega^2$. 
\item But one can "align" such couples of frames so that $\x$ and $\y$ are in the same $U$-orbit, that is $\x U=\y U$ (Lemma \ref{UGeneric}).
 \end{itemize}

 These   facts easily imply topological transitivity of $U$ on $\mathcal{E}$ 
(see section \ref{prooftoptrans}).\\
 

\subsubsection{Ergodicity of $\mu$ and $\lambda$}

 In the convex-cocompact case, the Patterson-Sullivan $\nu$ is 
Ahlfors-regular of dimension $\delta_\Gamma$. To go beyond that case, we will need 
to consider the lower dimension of the Patterson-Sullivan measure:

$$\underline{\dim} \, \nu =\textrm{infess}\liminf_{r\to 0}\frac{\log \nu(B(\xi,r))}{\log r},$$
which satisfies the following important property.
\begin{prop}[Ledrappier \cite{Ledrappier-Platon}]\label{dim} If $\mu$ is  finite, then $\underline{\dim} \, \nu=\delta_\Gamma $.
\end{prop}

 The first step in the proof of topological transitivity is the proof that
the closure of the set of $U$-orbits intersecting $\Omega$ is $\mathcal{E}$.  
The analogue here  is to show that for a $U$-invariant set $E$, 
it is sufficient to show that $\mu(E)=0$ or $\mu(E)=1$ to deduce that $\lambda(E)=0$ or $\lambda(E^c)=0$ respectively. Marstrand's  projection Theorem   and the hypothesis $\delta_\Gamma +\dim(U)>d-1$ allow  us to prove that the ergodicity of $\lambda$ is in fact equivalent to the ergodicity of $\mu$ (Proposition \ref{BMandBRequivalent}). Although it is highly unusual to study the ergodicity of non-quasi-invariant measures, 
it turns out here to be easier, thanks to finiteness of $\mu$. \\
 
 For the second step, we know thanks to Winter \cite{Winter} that 
the $A$-action on $(\Omega^2,\mu \otimes\mu)$ is mixing. So we can find couples $(\x,\y)\in \Omega^2$, 
which are typical in the sense that they satisfy   Birkhoff ergodic Theorem 
for the diagonal action of $A$ for negative times and continuous test-functions. 
By the same alignment argument than in the topological part, one can find such typical couples in the same $U$-orbit.\\
 
 Unfortunately, from the point of view of measures, existence of one individual orbit with some specified properties is meaningless. To circumvent this difficulty, we have to consider plenty of such typical couples on the same $U$-orbit. More precisely, we 
 consider a measure $\eta$ on $\Omega^2$ such that almost surely, 
a couple $(\x,\y)$ picked at random using $\eta$ is in the same $U$-orbit, 
and is typical for the diagonal $A$-action. For this to make sense when comparing 
with the measure $\mu$, we also require that both marginal laws of $\eta$ on 
$\Omega$ are absolutely continuous with respect to $\mu$. 
We check in section \ref{subsectionplenty} that the existence of such a 
measure $\eta$ is sufficient to prove Theorem \ref{ergodicity}. 
This measure $\eta$ is a kind of self-joining of the dynamical system $(\Omega,\mu)$, but instead of being invariant by a diagonal action, we ask that it 
reflects both the structure of $U$-orbits, and the mixing property of $A$.\\

 It remains to show that such a measure $\eta$ actually exists. In dimension $d=3$, we can construct it (at least locally on $\mathcal{F}\H^3$) as the direct image of $\mu\otimes \mu$ by the alignment map, so we present the simpler 3-dimensional case separately in section \ref{constrplentydim3}. The fact that $\eta$ is supported by typical couples on the same $U$-orbit is tautological from the chosen construction. The difficult part is to show that its marginal laws are absolutely continuous. This is a consequence of the following fact:\\
 
 {\em If two compactly supported, probability measures on the plane $\nu_1,\nu_2$ have finite $1$-energy, then for $\nu_1$-almost every $x$, the radial projection of $\nu_2$ on the unit circle around $x$ is absolutely continuous with respect to the Lebesgue measure on the circle.}\\
 
 Although probably unsurprising to the specialists, as there exists many related statements in the literature (see e.g. \cite{Mattila},\cite{MR1333890}), we were unable to find a reference. We prove this implicitly in our situation, using the $L^2$-regularity of the orthogonal projection in Marstrand's Theorem, and the maximal inequality of Hardy and Littlewood.\\
 
 In dimension $d\geq 4$, the construction of $\eta$, done in section \ref{higherdimeta}, is a bit more involved since there is not a unique couple aligned on the same $U$-orbit, especially if $\dim(U)\geq 2$, so we have to choose randomly amongst them, using smooth measures on Grassmannian manifolds. Again, the absolute continuity follows from Mastrand's projection Theorem and the maximal inequality.

\subsection{Organization of the paper}

Section 2 is devoted to introductory material. In section 3, we prove our results on topological
dynamics. In section 4, we introduce the measures $\mu$ and $\lambda$, establish the dimensional properties that we need, and prove Theorem \ref{divergence} and the fact that $U$-ergodicity of $\mu$ and $\lambda$ are
equivalent. Finally, we prove Theorem \ref{ergodicity} in section 5. 



\section{Setup and Notations}

\label{section2}


\subsection{Lie groups, Iwasawa decomposition}

Let $d \geq 2$, and $G=\SO^o (d,1)$, i.e. the subgroup of $\SL(d+1,\R)$ 
preserving the quadratic form $q(x_1,..,x_{d+1})=x_1^2+x_2^2+..-x_{d+1}^2$. 
It is the group of direct isometries of the hyperbolic $n$-space $\mathbb{H}^d=\{x\in\R^{d+1}, q(x)=-1, x_{d+1}>0\}$. 
Define $K<G$ as
\[
   K= \left\lbrace 
\left(  \begin{array}{cc}
 k & 0 \\
 0 & 1 \\
 \end{array} \right) : k \in \SO(d) \right\rbrace.
\]
It is a maximal compact subgroup of $G$, and it is the stabilizer of  the origin $x=(0,...0,1)\in \H^d$. 

 We choose the one-dimensional Cartan subgroup $A$, defined by
\[ 
 A = \left\lbrace a_t = \left(\begin{array}{cc}
   I_{d-2} & 0 \\
0 & \begin{array}{cc}
 \cosh(t) & \sinh(t) \\
 \sinh(t) & \cosh(t) \\
 \end{array}   \\
  \end{array}  \right) \, : \, t\in \R \right\rbrace.
\]
It commutes with the following subgroup $M$, which can be identified with $\SO(d-1)$.
\[  M=\left\lbrace 
\left(  \begin{array}{cc}
 m & 0 \\
 0 & I_2 \\
 \end{array} \right) : m \in \SO(d-1) \right\rbrace.
\]
In other words, the group $M$ is the centralizer of $A$ in $K$. 
The stabilizer of any vector $v\in T^1\H^d$ identifies with a conjugate of $M$, so that 
$T^1\H^d=\SO^o(d,1)/M$. 

 Let $\mathfrak{n}\subset \mathfrak{so}(d,1)$ be the eigenspace of $Ad(a_t)$ with eigenvalue $e^{-t}$. 
Let 
$$
N=\exp(\mathfrak{n}) 
\,.$$ 

It is an abelian, maximal unipotent subgroup, normalized by $A$.
The group $G$ is diffeomorphic to the product $K\times A\times N$. This decomposition is the {\em Iwasawa decomposition} of the group $G$.
\medskip
 
 The subgroup $N$ is normalized by $M$, and $M\ltimes N$ is a closed subgroup isomorphic 
to the orientation-preserving affine isometry group of an $d-1$-dimensional Euclidean space.
\medskip
 
 If $U$ is any closed, connected unipotent subgroup of $G$, 
it is conjugated to a subgroup of $N$ (see for example \cite{MR3012160}).
 Therefore, it is isomorphic to $\R^k$, for some $d \in \{ 0,.., d-1\}$. Through the article, we will always assume that $k\geq 1$.

In this paper, we are interested in the dynamical properties of the right actions of the subgroups
$A,N,U$ on the space $\Gamma\backslash G$. 


\subsection{Geometry}


\subsubsection*{Fundamental group, critical exponent, limit set}

 Let  $\Gamma \subset G=\mbox{Isom}^+(\H^d)$ be a discrete group. 
Let $\mathcal{M}= \Gamma \backslash \H^d$ be the corresponding hyperbolic manifold. 
The limit set $\Lambda_\Gamma$ is the set of accumulation points in $\partial \H^n \simeq \mathbb{S}^{d-1}$ of any orbit $\Gamma o$, where $o \in \H^d$. 
We will always assume that the group $\Gamma$ is nonelementary, that is $\#\Lambda_\Gamma=+\infty$.

\medskip
The critical exponent $\delta$ of the group $\Gamma$ is the infimum of the $s>0$ such that the Poincar\'e series
$$P_\Gamma(s)=\sum_{\gamma\in \Gamma} e^{-sd(o,\gamma o)},$$
is finite, where $o$ is the choice of a fixed point in $\H^d$. In the convex-cocompact case, 
the critical exponent $\delta$ equals the Hausdorff dimension of the limit set $\Lambda_\Gamma$.
 Since $\Gamma$ is non-elementary, we have $0<\delta\leq d-1$.


\subsubsection*{Frames}

 The space of orthonormal, positively oriented frames over $\H^d$  (resp. $\mathcal{M}$) will be denoted by $\mathcal{F}\H^d$ (resp.  $\mathcal{FM}$). 
As $G$ acts simply transitively on $\F\H^d$,  $\mathcal{F}\H^d$ (resp.  $\mathcal{FM}$) 
  can be identified with $G$  (resp. $\Gamma\backslash G$)  by the map  $g\mapsto g.\x_0$, where $\x_0$ is a fixed reference frame. 
Note that $\F\H^d$ is a $M$-principal bundle over $T^1\H^d$, and so is $\F\mathcal{M}$ over $T^1\mathcal{M}$. 
Denote by $\pi_1:\F \mathcal{M} \to T^1\mathcal{M}$ (resp. $\F\H^d\to T^1\H^d$) the projection of a frame onto its first vector. 

\medskip

As said above, we are interested in the properties of the right actions of $A,N,U$ on $\mathcal{FM}$. 

\begin{figure}[ht!]\label{groupactions}
\begin{center}
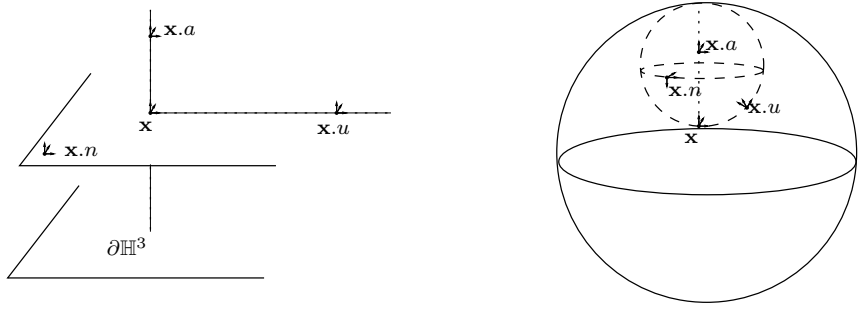
\caption{Right actions of $A$, $N$, $U$ } 
\end{center}
\end{figure} 

\medskip 

Given a subset $E\subset \mathcal{M}$ (resp. $T^1 \mathcal{M}$, $\F \mathcal{M}$), we will write $\tilde{E}$ for its lift to $\H^d$  (resp. $T^1 \H^d$, $\F \H^d$). 

\medskip 
 Denote by $\F \mathbb{S}^{d-1}$ the set of (positively oriented) frames over $\partial \H^d=\mathbb{S}^{d-1}$.
 We will write $\F \Lambda_\Gamma$ for the subset of frames which are based at $\Lambda_\Gamma$.

\subsubsection*{Generalised Hopf coordinates}
Choose $o$ to be the point $(0,\dots, 0, 1)\in\H^d$. 
Recall that the Busemann cocycle is defined on $\mathbb{S}^{d-1}\times \H^d\times \H^d$ by
\[\beta_\xi(x,y)=\lim_{z\to \xi} d(x,z)-d(y,z)\]
 By abuse of notation, if $\x, \x'$ are frames (or $v,v'$ vectors) with basepoints $x,x'\in\H^d$,  
 we will write $\beta_\xi(\x,\x')$ or $\beta_\xi(v,v')$ 
for  $\beta_\xi(x,x')$. 

 We will use the following extension of the classical Hopf coordinates to describe frames. To a frame $\x \in \F\H^d$, we associate
\begin{align*}
\F\H^d & \to (\F \mathbb{S}^{d-1} \times_\Delta \mathbb{S}^{d-1})  \times \R,\\
 \x=(v_1,\dots, v_d) & \mapsto (\x^+,x^-,t_\x),
 \end{align*}
where $x^-$ (resp. $x^+$) is the negative (resp. positive) endpoint in $\mathbb{S}^{d-1}$ of the geodesic $\x A$, $t_\x=\beta_{x^+}(o,\x)$,  
 and $\x^+\in \F\mathbb{S}^{d-1}$ is the frame over $x^+$ obtained
for example by parallel transport along $\x A$ of the $(d-1)$-dimensional frame $(v_2,\dots, v_n)$. The subscript $\Delta$ in $(\F \mathbb{S}^{d-1} \times_\Delta \mathbb{S}^{d-1})$ indicates that this is the product set, minus the diagonal, i.e. the set of $(\x^+,x^-)$ where $\x^+$ is based at $x^-$.

\begin{figure}[ht!]\label{hopf}
\begin{center}
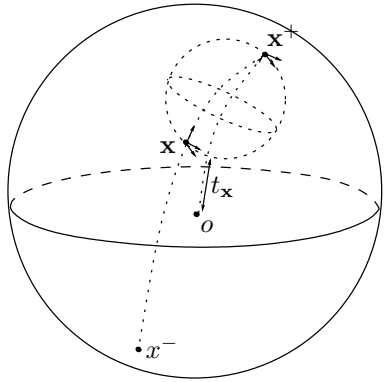
\caption{Hopf frame coordinates } 
\end{center}
\end{figure} 
 
 \medskip
 
 Define the following subsets of frames in Hopf coordinates
 \[
  \widetilde{\Omega}=(\F \Lambda_\Gamma \times_\Delta \Lambda_\Gamma)  \times \R,
 \] 
 and 
 \[
  \tilde{\mathcal{E}} = (\F \Lambda_\Gamma \times_\Delta \partial \H^n)   \times \R.
 \]
  Consider their quotients $\Omega=\Gamma \backslash \widetilde{\Omega}$ and $\mathcal{E}=\Gamma \backslash \tilde{\mathcal{E}}$. 
These are  closed invariant subsets of $\F \mathcal{M}$ for the dynamics of $M\times A$ and $(M\times A) \ltimes N$ respectively, 
where all the dynamic happens. Let us state it more precisely. 
\medskip

The {\em non-wandering set} of the action of $N$ (resp. $U$) on $\Gamma\backslash G$ is the set of frames
$\x\in \F\M$ such that given any neighbourhood $\mathcal{O}$ of $\x$ there exists a sequence $n_k\in N$ (resp. $u_k\in U$) going to $\infty$
such that $n_k \mathcal{O}\cap \mathcal{O}\neq\emptyset$. 
As a consequence of Theorem  \ref{topologicallytransitive}, the following result holds. 
\begin{prop} The set $\mathcal{E}$ is the nonwandering set of $N$ and of any unipotent subgroup $\{0\}\neq U<N$. 
\end{prop}


\section{Topological dynamics of geodesic and unipotent frame flows}

\label{sectiontop}

\subsection{Dense leaves and periodic vectors} 
For the proof of Theorem \ref{Atopmix}, we will need the following intermediate result, of independent interest . 

\begin{prop} \label{densityperiodic}
Let $\Gamma$ be a Zariski-dense subgroup of $\SO^o(d,1)$.  
Let {\rm $\x\in\Omega$} be a frame such that {\rm $\pi_1(\x)$} is a periodic orbit
of the geodesic flow on $T^1\M$. Then its $N$-orbit {\rm $\x N$} is dense in $\mathcal{E}$. 
\end{prop} 

\begin{proof} First, observe that if $v=\pi_1(\x) \in T^1\M$ is a periodic vector for the geodesic flow, 
then its strong stable manifold $W^{ss}(v)$ is dense in 
$\pi_1(\mathcal{E})$  \cite[Proposition B]{MR1779902}. 

Therefore, $\pi_1^{-1}(W^{ss}(\pi_1(\x)))=\x NM=\x MN$ is dense in $\mathcal{E}$. 
Thus it is enough to prove that
\[
\x M \subset \overline{\x N}.
\]
The crucial tool is a Theorem of Guivarc'h and Raugi \cite[Theorem 2]{MR2339285}. 
We will use it in two different ways depending if $G=\SO^o(3,1)$ or $G=\SO^o(d,1)$, for $d\ge 4$, 
the reason being that $M=\SO(d-1)$ is abelian in the case $d=3$. 

Choose $\tilde{\x}$ a lift of $\x$ to $\widetilde{\Omega}$. As $\pi_1(\x)$ is periodic, 
say of period $l_0>0$, but $\x$ itself 
has no reason to be periodic, there exists $\gamma_0\in \Gamma$ and $m_0\in M$ such that
\[
\tilde{\x}a_{l_0}m_0=\gamma_0\tilde{\x}. 
\]

First assume $d=3$, so both $M$ and $MA$ are abelian groups. 
Let $C$ be the connected compact abelian group  $C=MA/\langle a_{l_0}m_0 \rangle$. 
Let $\rho$ be the homomorphism from $MAN$
to $C$ defined by $\rho(man)=ma$ mod $\langle a_{l_0}m_0 \rangle$. 
Define $X^\rho=G\times C/\sim$, where $(g,c)\sim (gman, \rho(man)^{-1}.c)$. 
The set $X^\rho$ is a fiber bundle over $G/MAN=\partial \H^n$, whose fibers are isomorphic to $C$. 
In other terms, it is an extension of the boundary containing additional information on how $g$ is positioned along $AM$, modulo $a_{l_0}m_0$. 
Let $\Lambda_\Gamma^\rho$ be the preimage of $\Lambda_\Gamma\subset \partial \H^n$ inside $X^\rho$. 
Now, since $C$ is connected, \cite[Theorem 2]{MR2339285} asserts that the action of $\Gamma$ on $\Lambda_\Gamma^\rho$ is minimal. 
Denote by $[g,m]$ the class of $(g,m)$ in $X^\rho$.

\medskip
Let us deduce that $\x M \subset \overline{\x N}$. Choose some $m\in M$. 
As $\Gamma$ acts minimally on $\Lambda_\Gamma^\rho$, 
there exists a sequence $(\gamma_k)_{k \geq 1}$ of elements of $\Gamma$, such
that 
$\gamma_k[\tilde{\x},e]$ converges to $[\tilde{\x}m,e]$. It means that there exist sequences $(m_k)_k\in M^\N$, $(a_k)_k\in A^\N$, $(n_k)_k\in N^\N$, 
such that $\gamma_k\tilde{\x}m_ka_kn_k \to \tilde{\x}m$ in $G$, whereas $\rho(m_ka_kn_k)\to e$ in $C$, which means that there exists some sequence
$j_k$ of integers, such that $d_k:=(m_ka_k)^{-1}(a_{l_0}m_0)^{j_k}\to e $ in $MA$. 

\medskip
Now observe that the sequence 
\[ 
\gamma_k\tilde{\x}(a_{l_0}m_0)^{j_k}(d_k^{-1}n_kd_k)=(\gamma_k\gamma_0^{j_k})\tilde{\x}(d_k^{-1}n_kd_k)\in \Gamma \tilde{\x} N
 \]
 has the same limit as the sequence 
 \[
 \gamma_k\tilde{\x}(a_{l_0}m_0)^{j_k}d_k^{-1}n_k=\gamma_k\tilde{\x}m_ka_kn_k,
\] which   by construction converges to  $ \tilde{\x}m$. 
On $\F\M=\Gamma\backslash G$, it proves precisely that $\x m\in \overline{\x N}$. 
As $m$ was arbitrary, it concludes the proof in the case $n=3$. \\

In dimension $d\ge 4$,  $\langle a_{l_0}m_0 \rangle$ is not always a normal subgroup of $MA$ anymore, so
we have to modify the argument as follows. 

Denote by $M_\x$ the set 
$$M_\x=\{m\in M, \x m\in\overline{\x N}\}\,.$$
 This is a closed subgroup of $M$; indeed, if $m_1,m_2 \in M_\x$, then $\x m_1 \in \overline{\x N}$, so $\x m_1m_2 \in \overline{\x N}m_2=\overline{\x m_2 N}$ since $m_2$  
normalises $N$. Since $\x m_2 \in \overline{\x N}$, we have $\x m_2 N\subset \overline{\x N}$. So $\x m_1 m_2 \in \overline{\x m_2 N}\subset \overline{\x N}$. Thus $M_\x$ is a subsemigroup, non-empty since it contains $e$, and closed. Since $M$ is a compact group, such a closed semigroup is automatically a group.\\

We aim to show that the group $M_\x$ is necessarily equal to $M$.\\

Let $C=MA/\langle a_{l_0} \rangle$. It is a compact connected group. Consider $\rho(man)=ma$ mod $\langle a_{l_0} \rangle$, and
the associated boundary $X^\rho=G\times C/\sim$. Choose some $m\in M$. As above, \cite[Theorem 2]{MR2339285} asserts that the action of $\Gamma$ on $\Lambda_\Gamma^\rho$ is minimal.  
Therefore, there exists a sequence $(\gamma_k)_{k \geq 1}$ of elements of $\Gamma$, such
that 
$\gamma_k[\tilde{\x},e]$ converges to $[\tilde{\x}m,e]$. As above, consider  sequences $(m_k)_k\in M^\N$, $(a_k)_k\in A^\N$, $(n_k)_k\in N^\N$,  
such that $\gamma_k\tilde{\x}m_ka_kn_k \to \tilde{\x}m$ in $G$, whereas $\rho(m_ka_kn_k)\to e$ in $C$, which 
with this new group $C$ means that there exists some sequence
$j_k$ of integers, such that $d_k:=(m_ka_k)^{-1}(a_{l_0})^{j_k}\to e $ in $MA$.

 Similarly to the $3$-dimension case, we can write
 \[ 
\gamma_k\tilde{\x}m_ka_kn_k d_k= \gamma^k \tilde{\x} a_{l_0}^{j_k}(d_k^{-1} n_k d_k)= (\gamma_k \gamma_0^{j_k}) \tilde{\x} m_0^{-j_k} (d_k^{-1} n_k d_k)
 \]
 
The above argument shows that some sequence of frames in $  \x  \langle m_0 \rangle N=\x N \langle m_0 \rangle$ converges to 
$\x m$. This implies that the set of products $M_\x.\overline{\langle m_0\rangle} $ is equal to $M$. 

We use a dimension argument to conclude the proof. The group
 $\overline{ \langle m_0 \rangle}$ is a torus inside $M=\SO(d-1)$, therefore of dimension at most $\frac{d-1}{2}$. The group 
$M$ has dimension $\frac{(d-1)(d-2)}{2}$, so that $M_\x.\overline{\langle m_0\rangle}=M$ implies that $\dim M_\x\ge \frac{(d-1)(d-3)}{2}$.
By \cite[lemma 4]{montgomery-samelson}, the dimension of any proper closed subgroup of $M=\SO(d-1)$ is smaller than $\dim \SO(d-2)=\frac{(d-2)(d-3)}{2}$. 
Therefore, $M_\x$ cannot be a proper subgroup of $M$, so that $M_\x=M$.

\end{proof}

The following corollary is a generalization to $\F \M$ of a well-known result on $T^1\M$, due to Eberlein. 
A vector $v\in T^1\M$ is said quasi-minimizing if there exists a constant $C>0$ such that for all $t\ge 0$, 
$d(g^t v,v)\ge t-C$. In other terms, the geodesic $(g^t v)$ goes to infinity at maximal speed. 
We will say that a frame $\x\in \F\M$ is quasi-minimizing if its first vector $\pi_1(\x)$ is quasi-minimizing.

\begin{coro}\label{density-N-orbites-pas-quasi-minimisantes} Let $\Gamma$ be a Zariski dense subgroup of   $G= \SO^o(d,1)$ . 
A frame {\rm $\x\in \Omega$} is not quasi-minimizing if and only if {\rm $\x N$} is dense in $\mathcal{E}$. 
\end{coro}

\begin{proof} First, observe that when $\x\in\Omega$ is quasi-minimizing, 
then the strong stable manifold $W^{ss}(\pi_1(\x))$ of its
first vector is not dense in $\pi_1(\Omega)$.  
Therefore, $\x N\subset \pi_1^{-1}(W^{ss}(\pi_1(\x))$ cannot be dense in $\Omega$. 

Now, let $\x\in \Omega$ be a non quasi-minimizing vector.
 Then $W^{ss}(\pi_1(\x))$ is  dense in $\pi_1(\Omega)$, 
so that $\x NM=\x MN=\pi_1^{-1}(W^{ss}(\pi_1(\x))$
is dense in $\Omega$, and therefore in $\mathcal{E}=\Omega N$. 
 Choose some $\y\in\Omega$ such that $\pi_1(\y) $ is a periodic orbit of the geodesic flow. 
By the above proposition, $\y N$ is dense in $\mathcal{E}$. 
As $\x NM$ is dense in $\mathcal{E}\supset \Omega$, we have $\y M\subset \overline{\x NM}=\overline{\x N}M$ (this last equality following from the compactness of $M$), so that there exists $m\in M$ with $\y m\in\overline{\x N}$. 
But $\pi_1(\y m)=\pi_1(\y)$ is periodic, so that $\y mN$ is dense in $\mathcal{E}$ and $\overline{\x N}\supset \overline{\y mN}\supset \mathcal{E}$. 
\end{proof}

\subsection{Topological Mixing of the geodesic frame flow}

 Recall that the continuous flow $(\phi_t)_{t\in \R}$ (or a continuous transformation $(\phi_k)_{k\in \Z}$) 
on the topological space $X$ is {\em topologically mixing} if for any two non-empty open sets $\mathcal{U},\mathcal{V} \subset X$, 
there exists $T>0$ such that for all $t>T$, 
 \[
 \phi_{-t} \mathcal{U}\cap \mathcal{V} \neq \emptyset. 
 \]

Let us now prove Theorem \ref{Atopmix}, by a refinement of an argument of Shub also used by Dal'bo \cite[p988]{MR1779902}.
\begin{proof} We will proceed by contradiction and assume that the action of $A$ is not mixing. 
Thus there exist $\mathcal{U},\mathcal{V}$  two non-empty open sets in $\Omega$, and a sequence $t_k\to +\infty$, 
such that $\mathcal{U}.a_{t_k}\cap \mathcal{V}=\emptyset$. 
Choose $\x\in \mathcal{V}$ such that $\pi_1(\x)$ is periodic for the geodesic flow - 
this is possible by density of periodic orbits in $\pi_1(\Omega)$ \cite[Theorem 3.10]{MR0310926}. 
Let $l_0>0, m_0\in M$ be such that $\x a_{l_0}m_0=\x$. 

\medskip 
 In particular, we have $\x a_{l_0}^{j}=\x m_0^{-j}\in \x G_{\varepsilon}$ for all $j\in J$. 
We can find integers $(j_k)_k$ (the integer parts of $t_k/l_0$) and real numbers $(s_k)_k$ such that:
$$
t_k=j_kl_0+s_k, \, \mathrm{with} \;  0\le s_k< l_0. 
$$

 Without loss of generality, we can assume that the sequence $(s_k)_{k\geq 0}$ converges to some $s_\infty\in [0,l_0]$, and that $m^{j_k}$ converge in the compact group $M$ to some $m_\infty\in M$.
By Proposition \ref{densityperiodic}, the $N$-orbit 
$\x a_{-s_\infty}m_\infty N$ is dense in $\mathcal{E}$. Notice that $\mathcal{U}N$ is an open subset of $\mathcal{E}$; 
therefore one can choose a point $\w=\x a_{-s_\infty}m_\infty n\in \mathcal{U}$, for some $n \in N$. 

\begin{figure}[ht!]\label{mixing-figure}
\begin{center}
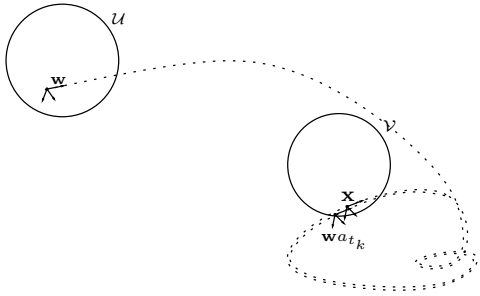
\caption{The frame flow is mixing } 
\end{center}
\end{figure} 
We have

\begin{align*}
\w a_{t_{k}} & = \x a_{-s_\infty}m_\infty n a_{t_{k}} \\
             & =\x ( a_{l_0}m)^{j_k} (m^{-j_k} m_\infty) (a_{s_k-s_\infty}) (a_{t_k} n a_{-t_k})\\
             & = \x (m^{-j_k} m_\infty) (a_{s_k-s_\infty}) (a_{t_k} n a_{-t_k}).
\end{align*}

Observe that, as $N$-orbits are strong stable manifolds for the $A$-action, so 
$$\lim_k a_{t_k} n a_{-t_k}=e.$$
By definition of $m_\infty$ and $s_\infty$,  $\lim_k m^{-j_k} m_\infty=e$ and $\lim_k a_{s_k-s_\infty}=e$.
Therefore, $\w a_{t_{k}}$ tends to the frame $\x$ in the open set $\mathcal{V}$. Thus, we found a frame $\w\in \mathcal{U}$, with $\w a_{t_{k}}\in \mathcal{V}$ for all $k$ large enough. Contradiction. 
\end{proof}

\subsection{Dense orbits for the diagonal frame flow on $\Omega^2$}

 Recall that a continuous flow $(\phi_t)_{t\in \R}$ (or a continuous transformation $(\phi_k)_{k\in \Z}$) 
on the topological space $X$ is said to be {\em topologically transitive} if any nonempty invariant open set is dense.  

\medskip
In the case of a continuous transformation on a complete separable metric space without isolated points, 
topological transitivity is equivalent to the existence of a dense positive orbit, or equivalently, 
that the set of dense positive orbits is a $G^\delta$-dense set (see for example \cite{MR1973093}).

 \medskip
  It is clear that topological mixing implies topological transitivity. Moreover, as is easily checked, topological mixing 
of $(X,\phi_t)$ implies topological mixing for the diagonal action on the product $(X \times X,(\phi_t,\phi_t))$. 

 \medskip 
 A couple $(\x,\y) \in \Omega^2$ will be said  {\em generic} if the negative diagonal, 
discrete-time orbit $(\x a_{-k} ,\y a_{-k} )_{k\geq 0}$ 
is dense in $\Omega^2$. Theorem \ref{Atopmix} about topological mixing of the $A$-action on $\Omega$ has   the following corollary,
 which will be useful in the proof of Theorem \ref{topologicallytransitive}.

 \begin{coro} \label{densediagonal}
 If $\Gamma\subset G={\SO^o(d,1)}$ is a Zariski-dense discrete subgroup, 
then there exists a generic couple $(\textbf{\rm \bf x},\textbf{\rm \bf y}) \in \Omega^2$.
 \end{coro}
 \begin{proof} By Theorem \ref{Atopmix}, the geodesic frame flow is topologically mixing. 
Therefore, so is the diagonal flow action 
of $A$ on $\Omega^2$. This implies that the transformation $(a_{-1},a_{-1})$  
on $\Omega^2$ is also topologically mixing, hence topologically transitive, i.e. has a dense positive orbit.
 \end{proof}


 \subsection{Existence of a generic couple on the same $U$-orbit}

 \begin{lem} \label{UGeneric}
  There exists a generic couple of the form $(\textbf{\rm \bf x} ,\textbf{\rm \bf x}  u)$, 
with $\textbf{\rm \bf x}\in \Omega$ and $u\in U$.
 \end{lem}
 \begin{proof} By Corollary \ref{densediagonal}, there exists a generic couple.

  \medskip
Let $(\y,\z)\in (\F\H^d)^2$ be the lift of a generic couple. 
Notice that, since the actions of $A$ and $M$ commute  with $A$, 
the set of generic couples is invariant under the action of $(A\times M) \times (A\times M)$. This means that in Hopf coordinates, 
being the lift of a generic couple does not depend on the orientation of the frame $\y^+,\z^+$, nor of the times $t_\y,t_\z$. 
Moreover, since being generic is defined as density for {\em negative} times, one can also freely change the base-points of $\y^+,\z^+$ 
because the new negative orbit will be exponentially close to the old one. In short, being the lift of a generic couple 
(or not) depends only on the past endpoints $(y^-,z^-)$, or equivalently, is $\left( (M\times A)\ltimes N^- \right)^2$-invariant. 
Obviously, $y^-\neq z^-$ since generic couple cannot be on the diagonal.
  
  \medskip 
  Up to conjugation by elements of $M$, we can freely assume that $U$ contains the subgroup corresponding to following the 
direction given by the second vector of a frame. 
  Pick a third point $\xi \in \Lambda_\Gamma$ distinct from $y^-$ and $z^-$, and choose a frame $\x^+\in \F \Lambda_\Gamma$ 
based at $\xi$, whose first vector is tangent to the circle determined by $(\xi,y^-,z^-)$. Therefore, the two frames of Hopf
 coordinates $\x=(\x^+,y^-,0)$ and $( \x^+,z^-,0)$ lie in the same $U$-orbit, thus $(\x^+,z^-,0)=\x u$ for some $u \in U$. 
By construction, the couple $( \x,\x u)$ is the lift of a generic couple.
 \end{proof}

\subsection{Minimality of $\Gamma$ on $\F\Lambda_\Gamma$}

 We recall the following known fact. 

 \begin{prop} \label{minimalboundaryset}
  Let $\Gamma$ be a Zariski-dense subgroup of ${\SO}_o(d,1)$. Then the action of $\Gamma$ on $\F \Lambda_\Gamma$ is minimal.
 \end{prop}

 In dimension $d=3$, this is due to Ferte \cite[Corollaire E]{MR1934285}. 
In general, this is again a consequence of Guivarc'h-Raugi \cite[Theorem 2]{MR2339285}, applied with $G=SO_o(d,1)$, $C=M$.
 The set $\F\H^d$ identifies with $G\times M/\sim$ where $(g,m)\sim (gm'an,m'^{-1}m)$. \cite[Theorem 2]{MR2339285}  asserts
that the $\Gamma$-action on $\F\H^d=G\times M/\sim$ has a unique minimal set, which is necessarily $\F\Lambda_\Gamma$.  
 
 
\subsection{Density of the orbit of $\Omega$}
 
 \begin{prop}\label{UOmegadense}
The $U$-orbit of $\Omega$ is dense in $\mathcal{E}$.
 \end{prop}
\begin{proof}

 Up to conjugation 
by an element of $M$, it is sufficient to prove the proposition in the case where $U$ contains 
the subgroup corresponding to shifting in the direction of the first vector of the frame $\x^+$.
 \medskip

 Consider the subset $E$ of $\F \Lambda_\Gamma$ defined by $(\xi,R) \in E$ if $\xi \in \Lambda_\Gamma$ and there exists
 a sequence $(\xi_n)_{n\geq 0} \subset \Lambda_\Gamma \setminus\{ \xi \}$ such that $\xi_n \to \xi$ tangentially to the 
direction of the first vector of $R$, in the sense that the direction of the geodesic (on the sphere $\partial \H^n$) 
from $\xi$ to $\xi_n$ converges to the
direction of the first vector of $R$.  Clearly, $E$ is a non-empty, $\Gamma$-invariant set. By Proposition \ref{minimalboundaryset},
 it is dense in $\F \Lambda_\Gamma$.
\medskip

 Let $\x$ be a frame in $\tilde{\mathcal{E}}$, we wish to find a frame arbitrarily close to $\x$, which is in the $U$-orbit of $\widetilde{\Omega}$. 
Let $\x=(\x^+,x^-,t_\x)$ be its Hopf coordinates, by assumption $\x^+ \in \F \Lambda_\Gamma$. 
Pick $(\xi,R) \in E$ very close to  $\x^+$. By definition of $E$, there exist $\xi' \in \Lambda$, 
very close to $\xi$ such that the direction $(\xi\xi')$ is close to the first vector of the frame $R$. 
We can find a frame $\y^+\in \F \Lambda_\Gamma$, based at $\xi$, close to $\x^+$, whose first vector is tangent to the circle going through $(\xi,\xi',x^-)$.
 
\begin{figure}[ht!]\label{ }
\begin{center}
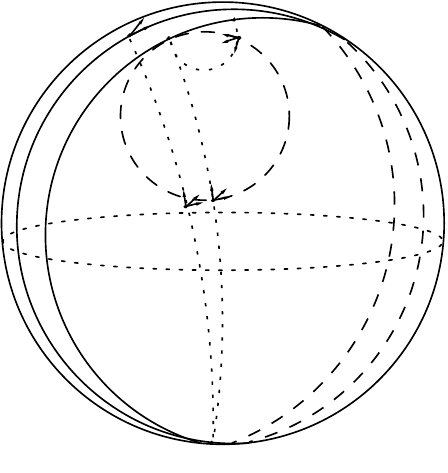
\caption{  } 
\end{center}
\end{figure} 
 \medskip

 By construction, the two frames $\y=(\y^+,x^-,t_\x)$ and $\z=(\y^+,\xi',t_\x)$ belong to the same $U$-orbit; 
notice that $\z\in \widetilde{\Omega}$, so we have $\y\in \widetilde{\Omega}U$. Since $\y^+$ and $\x^+$ are arbitrarily close, so are $\x$ and $\y$.
\end{proof}

\subsection{Proof of Theorem \ref{topologicallytransitive}}
\label{prooftoptrans}

 Let $\mathcal{O},\mathcal{O}' \subset \mathcal{E}$ be non-empty open sets. 
We wish to prove that $\mathcal{O}'U\cap \mathcal{O}U\neq \emptyset$. 
By Proposition \ref{UOmegadense}, $\mathcal{O}\cap \Omega U \neq \emptyset$, 
therefore $\mathcal{O}U \cap \Omega$ is an open nonempty subset of $\Omega$. 
Similarly, $\mathcal{O}'U\cap \Omega\neq \emptyset$.
 
 \medskip

 Let $(\x,  \x u)$ a generic couple given by Lemma \ref{UGeneric}. By density, there exists a $k\geq 0$ such that $(\x a_{-k},  \x ua_{-k}) \in 
(\mathcal{O}U \cap \Omega)\times (\mathcal{O}'U\cap \Omega)$. But since $A$ normalizes $U$, $\x ua_{-k} \in \x a_{-k}U\subset \mathcal{O}U$. 
Therefore 
$\x ua_{-k} \in \mathcal{O}'U\cap \mathcal{O}U$, which is thus non-empty, as required.


\section{Mesurable dynamics}

\label{section4}

\subsection{Measures}

Let us introduce the measures that will play a role here. \\

The {\em Patterson-Sullivan measure on the limit set} is a measure $\nu $ on the boundary, whose support is $\Lambda_\Gamma$, 
which is quasi-invariant under the action of $\Gamma$, and more precisely satisfies 
for all $\gamma\in\Gamma$ and $\nu$-almost every $\xi\in\Lambda_\Gamma$, 
\[\frac{d\gamma_*\nu}{d\nu}(\xi)=e^{-\delta \beta_\xi(o,\gamma o)}\,.
\] 

When $\Gamma$ is convex-cocompact, this measure is proportional to the Hausdorff measure of the limit set \cite{Sullivan84}, 
it is the intuition to keep in mind here. \\

On the unit tangent bundle $T^1\H^d$, let us define a $\Gamma$-invariant measure by 
$$
d\tilde{m}_{BM}(v)=e^{\delta\beta_{v^-}(o,v)+\delta\beta_{v^+}(o,v)}d\nu(v^-)d\nu(v^+)dt^,.
$$ 
By construction, this measure 
is invariant under the geodesic flow 
and induces on the quotient on $T^1\mathcal{M}$ the so-called {\em Bowen-Margulis-Sullivan measure $m_{BMS}$}.
When finite, it is the unique measure
of maximal entropy of the geodesic flow, and is ergodic and mixing. 

On the frame bundle $\mathcal{F}\H^d$ (resp. $\F\M$), there is a unique way 
to define a $M$-invariant lift of the Bowen-Margulis measure, that we will denote by $\tilde{\mu}$ (resp. $\mu$). 
We still call it the {\em Bowen-Margulis-Sullivan measure}. 
On $\F\M$, this measure has support $\Omega$. 
When it is finite, it is ergodic and mixing \cite{Winter}. 
The key point in our proofs will be that it is mixing, and that it is locally equivalent to the product
$d\nu(x^-)d\nu(x^+)\,dt\,dm_\x$, where $dm_\x$ denotes the Haar measure on the fiber of $\pi_1(\x)$, for the fiber bundle $\F\M\to T^1\M$. 
This measure is $MA$-invariant, but not $N$-(or $U$)-invariant, nor even quasi-invariant. \\

The {\em Burger-Roblin measure} is defined locally on $T^1\H^d$ as 
\[d\tilde{m}_{BR}(v)=e^{ (d-1)\beta_{v^-}(o,v)+\delta\beta_{v^+}(o,v)}d\L(v^-)d\nu(v^+)dt\,,\] where $\L$ denotes the Lebesgue 
measure on the boundary $\mathbb{S}^{d-1}=\partial \H^d$, invariant under the stabiliser $K \simeq \SO(d)$ of $o$.  
We denote its $M$-invariant extension to $\F\H^d$ (resp. $\F\M$), still called the {\em Burger-Roblin measure}, by $\tilde{\lambda}$ (resp. $\lambda$). 
This measure is infinite, $A$-quasi-invariant, $N$-invariant. It is $N$-ergodic as soon as $\mu$ is finite.  This has been proven by Winter \cite{Winter}. 
See also \cite{Schapira2015} for a short proof that it is the unique $N$-invariant measure supported in $\mathcal{E}_{rad}$.  

In some proofs,  we will need to use the properties of the conditional measures of $\mu$ on the strong stable leaves of the $A$-orbits, that is 
the $N$-orbits. These conditional measures can easily be expressed as 
$$d\mu_{\x N}(\x n)=e^{\delta\beta_{(\x n)^-}(\x,\x n)}\,d\nu((\x n)^-), 
$$
and the quantity $e^{\delta\beta_{(\x n)^-}(x,\x n)}$ is equivalent to $|n|^{2\delta}$ when $|n|\to +\infty$. 

Observe also that by construction, the measure $\mu_{\x N}$ has full support in the set $\{\y\in \x N,  y^-\in\Lambda_\Gamma\}$. 

Another useful fact is that $\mu_{\x N}$ does not depend really on $\x$ in the sense that it comes from a measure on $\partial \H^n\setminus \{x^+\}$. 
In other terms, if $m\in M$ and $\y\in \x m N$, and $\z\in \x N$ is a frame with $\pi_1(\z)=\pi_1(\y)$, one has
$d\mu_{\x mN}(\y)=d\mu_{\x N}(\z)$.

\subsection{Dimension properties on the measure $\nu$}

Most results in this paper rely on certain dimension properties on $\nu$, allowing to use projection theorems due to Marstrand \cite{MR0063439}, 
and explained in the books of Falconer \cite{MR2118797} and Mattila \cite{MR1333890}.
These properties are easier to check in the convex-cocompact case, relatively easy in the geometrically finite, and more subtle in general, 
under the sole assumption that $\mu$ is finite.\\

Define the dimension of $\nu$, like in \cite{LL}, by
$$\underline{\dim} \, \nu =\textrm{infess}\liminf_{r\to 0}\frac{\log\nu(B(\xi,r))}{\log r}\,.$$

Denote by $g^t$ the geodesic flow on $T^1\mathcal{M}$. For $v\in T^1\mathcal{M}$,  let $d(v,t)$ be the distance between the base point of $g^t v$ and the point $\Gamma.o$. 

Proposition \ref{dim} in the introduction has been established by Ledrappier \cite{Ledrappier-Platon} when $\mu$ is finite.    
It is also an immediate consequence of Proposition \ref{dim-mesfinie} and Lemma \ref{dist} below, as it is well known that when the measure $\mu$ is finite, it is ergodic and conservative. 

\begin{prop}\label{dim-mesfinie}
If $\mu$-almost surely, we have $\frac{d(v,t)}{t}\to 0$, 
 then $\underline{\dim} \, \nu \ge \delta_\Gamma$. 

If $\mu$ is ergodic and conservative, then $\underline{\dim}\,\nu\le\delta_\Gamma$.
\end{prop}

\begin{proof} 
We will come back to the original proof of the Shadow Lemma, of Sullivan, and adapt it (the proof, not the statement) to our purpose. 
The Shadow $\mathcal{O}_o(B(x,R))$ of the ball $B(x,R)$ viewed from $o$ is the set $\{\xi\in\partial\H^d$, $[o\xi)\cap B(x,R)\neq\emptyset\}$. Denote by $\xi(t)$ the point at distance $t$ of $o$ on the geodesic $[o\xi)$. 
It is well known that for the usual spherical distance, a ball $B(\xi,r)$ in the boundary is comparable to a shadow $\mathcal{O}_o(B(\xi(-\log r),R))$. 
More precisely, there exists a universal constant $t_1>0$ such that for all $\xi\in\partial \H^d$ and $0<r<1$, one has 
$$
  \mathcal{O}_o(B(\xi(-\log r+t_1),1))
\subset B(\xi,r)\subset \mathcal{O}_o(B(\xi(-\log r-t_1),1))
$$
Denote by $d(\xi,t)$ the distance $d(\xi(t),\Gamma.o)$. 
By assumption (in the application this will be given by Lemma \ref{dist}), for $\nu$-almost all $\xi\in\partial\H^d$ and $0<r<1$ small enough, 
the quantity $d(\xi,-\log r\pm t_1 )\le t_1+d(\xi,-\log r)$ is negligible compared to $t=-\log r$. 
Let $\gamma\in \Gamma$ be an element minimizing this distance $d(\xi,t)$. It satisfies obviously
$|d(o,\gamma o)-t|\le d(\xi,t)$.
Observe that, by a very naive inclusion, using just $1\le 1+(C+1)d(\xi,t)$, 
$$
\mathcal{O}_o(B(\xi(t-t_1),1)\subset \mathcal{O}_o(B(\gamma.o, 1+d(\xi,t-t_1))
$$

Now, using the $\Gamma$-invariance properties of the probability measure $\nu$, 
and the fact that for $\eta\in \mathcal{O}_o(B(\gamma.o,1+d(\xi,t-t_1))$, the quantity 
$|-\beta_\eta(o,\gamma o)+d(o,\gamma.o)-2d(\xi,t) | $ is bounded by some
universal constant $c$, one can compute  
\begin{eqnarray*}
 \nu(B(\xi,r))&\le& \nu(\mathcal{O}_o(B(\gamma.o,1+d(\xi,t-t_1))))\\
&=&\int_{O_o\left(B(\gamma.o,1+d(\xi,t-t_1))\right)} e^{-\delta_\Gamma\beta_{\eta}(o,\gamma o) }d\gamma_*\nu(\eta)\\
&\le& e^{\delta_\Gamma c}\, e^{-\delta_\Gamma d(o,\gamma o)+ 2\delta_\Gamma d(\xi,t)}\gamma_*\nu(O_o(B(\gamma.o,1+d(\xi,t-t_1))))\\
&\le& e^{\delta_\Gamma c}\, e^{-\delta_\Gamma t+ 2\delta_\Gamma d(\xi,t)}
\end{eqnarray*} 
Recall that $t=-\log r$. 
Up to some universal constants, we deduce that 
\begin{eqnarray}\label{estimee-nu}\nu(B(\xi,r))\le r^{\delta_\Gamma}e^{2\delta_\Gamma \,d(\xi,\log r)}
\end{eqnarray}
It follows immediately that $\underline{\dim} \, \nu\ge \delta_\Gamma$. 

The other inequality follows easily from the classical version of Sullivan's Shadow Lemma, or from the well known fact that $\delta_\Gamma$ is the Hausdorff dimension of the radial limit set, which has full $\nu$-measure. 
\end{proof}
\begin{lem}\label{dist} The following  assertions are equivalent, and hold when $\mu$ is finite.  
\begin{itemize}
\item for $\mu$-a.e. {\rm $\x\in \F M$}, one has
{\rm $$
\lim_{t\to +\infty}\frac{d(\x,\x a_t)}{t}=0\,.
$$
}
\item for $\lambda$-a.e. {\rm $\x\in \F M$}, one has
 {\rm $$
\lim_{t\to +\infty}\frac{d(\x,\x a_t)}{t}=0\,.
$$}
\item for $m_{BM}$ or $m_{BR}$ a.e. $v\in T^1M$, one has
$$
\lim_{t\to +\infty}\frac{d(v,g^tv)}{t}=\lim_{t\to +\infty}\frac{d(v,t)}{t}=0\,.
$$  
\item $\nu$-almost surely, 
$$
\lim_{t\to +\infty} \frac{d(\xi(t),\Gamma.o)}{  t}=\lim_{t\to +\infty} \frac{d(\xi,t)}{  t}=0\,.
$$
\end{itemize}
\end{lem}
When $\Gamma$ is geometrically finite, a much better estimate is known thanks to Sullivan's logarithm law (see \cite{MR688349}, \cite{MR1327939}, \cite[Theorem 5.6]{MR2732978}), since the distance grows typically in a logarithmic fashion. However, this may not hold for geometrically infinite manifolds with finite $\mu$. In any case, the above sublinear growth is sufficient for our purposes.  

\begin{proof}  First, observe that all statements are equivalent.
Indeed, first, as $m_{BR}$ and $m_{BM}$ differ only by their conditionals on stable leaves, and the limit 
$d(v,g^t v)/t$ when $t\to +\infty$ depends only on the stable leaf $W^{ss}(v)$, this property holds (or not) equivalently for $m_{BR}$ and $m_{BM}$. 

Moreover, as $\F M$ is a compact extension of $T^1M$, this property holds (or not) equivalently for 
$\lambda$ on $\F M$ and $m_{BR}$ on $T^1M$ or $\mu$ on $\F M$ and $m_{BM}$ on $T^1M$. 

As this limit depends only on the endpoint $v^+$ of the geodesic, and not really on $v$, 
the product structure of $m_{BR}$ implies that this property holds true equivalently for $m_{BM}$-a.e. $v\in T^1M$ and $\nu$ almost surely on the boundary. 

Let us prove that all these equivalent properties indeed hold when $\mu$ is finite. 

Let $f(v)=d(v,1)-d(v,0)$. As the geodesic flow is $1$-lipschitz, this map is bounded, and therefore
$\mu$-integrable. Thus, $\frac{S_n f}{n}$ converges a.s. to $\int f\,d\mu$, and therefore 
$d(v,t)/t\to \int f\,d\mu$, $\mu$-a.s. 

It is now enough to show 
that  this integral is $0$. This would be obvious if we knew that the distance $d(v,0)$ is $\mu$-integrable. 

Divide $\Omega$ in annuli   $K_n=\{v\in T^1\mathcal{M},d(\pi(v),o)\in   (n,n+1)\}$, and set $B_n=T^1B(o,n+1)$.
If $a_n=\mu(K_n)$, we have $\sum_n a_n=1$. 

Observe that $\int f\,d\mu=\lim_{n\to \infty}\int_{B_n} f\,d\mu$.

It is enough to find a sequence $n_k\to +\infty $ such that these integrals are arbitrarily small.
Observe that 
\[
\int_{B_N} f(x)d\mu(x)=\int_ {g^1(B_N) } d(v,0)d\mu- \int_{B_N} d(v,0)d\mu
\]
But now, the symmetric difference between $g^1B_N$ and $B_N$ is included in $K_N\cup   K_{N+1}$. 
As $d(v,0)\le N+2$ in this union, we get 
\[
\left| \int_{B_N} f(x)d\mu(x) \right| \leq (N+2)(a_N+a_{N+1}).
\]
As $\sum a_n=1$, there exists a subsequence $n_k\to +\infty$, such that 
$(n_k+2)(a_{n_k}+a_{n_k+1})\to 0$. 
This proves the lemma. 
\end{proof}

\subsection{Energy of the measure $\nu$}

The $t$-energy of $\nu$ is defined as 
$$
I_t(\nu)=\int\int_{\Lambda^2}\frac{1}{|\xi-\eta|^{t}}d\nu(\xi)\,d\,\nu(\eta) \, .
$$
The finiteness of a $t$-energy is sufficient to get the absolute continuity of the projection of $\nu$ on almost every $k$-plane of dimension $k<t$.
However, a weaker form of finiteness of energy will be sufficient for our purposes, namely


\begin{lem}\label{energie-finie} For all $t<\underline{\dim} \,\nu$, there exists an increasing  sequence $(A_k)_{k\geq 0}$ such that $I_t(\nu_{|A_k})<\infty$, and $\nu(\cup_k A_k)=1$.
\end{lem}
\begin{proof} When $t<\dim\nu$, choose some $t<t'<\dim\nu$. One has, for $\nu$-almost all $x$, and $r$ small enough, $\nu(B(x,r))\le Cst. r^{t'}$. 
It implies the convergence of the integral 
$$\int_{\Lambda}\frac{1}{|\xi-\eta|^{t}}d\,\nu(\eta)=t\int_0^\infty\frac{\nu(B(\xi,r))}{r^{t-1}}\,dr<\infty
$$
Therefore, the sequence of sets 
$A_M=\{x\in\partial\H^n,\int_{\Lambda_\Gamma}\frac{1}{|\xi-\eta|^{t}}d\,\nu(\eta)\le M\}$ is an increasing sequence whose union has full measure. 
And of course, $I_t(\nu_{|A_M})<\infty$. 
\end{proof}

 It is interesting to know when the following stronger assumption of finiteness of energy is satisfied. In \cite{MO}, when $\dim N=2$ and $\dim U=1$,  Mohammadi and Oh used  the following:

\begin{lem} If $\Gamma$ is convex-cocompact and $\delta>d-1-\dim U$ then $I_{d-1-\dim U}(\nu)<\infty$. 
\end{lem}
 

\begin{proof} For $\xi\in\Lambda_\Gamma$, 
 define $A_k=\{\eta\in\partial \H^d,\,|\xi-\eta|\in ]2^{-k-1},2^{-k}]\}$, and compute
\[
\int_{\Lambda_\Gamma}\frac{1}{|\xi-\eta|^{\dim N-\dim U}}d\,\nu(\eta) \simeq  \sum_{k\in \N^*}  2^{k(\dim N-\dim U)}\nu(A_k)
\]
Denote by $\xi_{k\log 2}$ the point at distance $k\log 2$ of $o$ on the geodesic ray $[o\xi)$. 
As $\Gamma$ is convex-cocompact, $\Omega$ is compact, so that  $\xi_{k\log 2}$ 
is at bounded distance from $\Gamma o$. Sullivan' Shadow lemma implies that, 
up to some multiplicative constant, 
 $\nu(A_k)\le \nu(B(\xi,2^{-k}))\le Cst. 2^{-k\delta}$. 
We deduce that, up to multiplicative constants, 
\begin{eqnarray*}
\int_{\Lambda_\Gamma}\frac{1}{|\xi-\eta|^{\dim N-\dim U}} d\,\nu(\eta) &\le &  \sum_{k } 2^{k(\dim N-\dim U-(1-\varepsilon)k\delta)}
\end{eqnarray*}
If $\delta>\dim N-\dim U$, for $\varepsilon>0$ small enough, the above series converges, uniformly in $\xi\in\Lambda_\Gamma$, so that the integral 
$\int\int_{\Lambda_\Gamma^2}\frac{1}{|\xi-\eta|^{\dim N-\dim U}} d\,\nu(\eta)d\,\nu(\xi)$ is finite, 
 and the Lemma is proven. 
\end{proof}

 As mentioned before, the reason we have to be interested in these energies is the following version of Marstrand's projection theorem, see for example \cite[thm 9.7]{MR1333890}. 

\begin{theo} \label{projL2}
Let $\nu_1$ be a finite measure with compact support in $\R^m$, such that $I_t(\nu_1)<\infty$, for some $0<t<m$. 
For all integer $k<t$, and almost all $k$-planes $P$ of $\R^m$, the orthogonal projection $(\Pi_P)_*\nu_1$
of $\nu_1$ on $P$ is absolutely continuous w.r.t. the $k$-dimensional Lebesgue measure of $P$. Moreover, 
its Radon-Nikodym derivative satisfies the following inequality
$$
\int_{\mathcal{G}_k^m} \int_P \left(\frac{d (\Pi_P)_*\nu_1}{d\mathcal{L}_P}\right)^2 d\mathcal{L}_P d\sigma_k^m < c. I_k(\nu_1)
 $$
where $\sigma_k^m$ is the natural measure on the Grassmannian $\mathcal{G}_k^m$, invariant by isometry, and $c$ some constant depending only on $k$ and $m$.  
\end{theo}


\subsection{Conservativity/ Dissipativity of $\lambda$}

In this section, we aim to prove Theorem \ref{nonergodicity}.

The measure $\lambda$ is $N$-invariant (and $N$-ergodic), 
therefore, $U$-invariant for all unipotent subgroups $U<N$. 

It is $U$-conservative iff for all sets $E\subset \F\M$ with positive measure, 
and $\lambda$-almost all frames $\x\in\F\M$, the integral $\int_0^\infty {\bf 1}_E(\x u)du$ diverges,
 where $du$ is the Haar measure of $U$. 
In other words, it is conservative when it satisfies the conclusion of Poincar\'e recurrence theorem 
(always true for a finite measure). 

It is $U$-dissipative iff for all sets $E\subset \F\M$ with positive finite measure, 
and $\lambda$-almost all frames $\x\in\F\M$, the integral $\int_0^\infty {\bf 1}_E(\x u)du$ converges. 

A measure supported by a single orbit can be both ergodic and dissipative. In other cases, ergodicity implies conservativity \cite{Aar}. 
Therefore, Theorem \ref{ergodicity} implies that when the Bowen-Margulis-Sullivan measure is finite, and $\delta_\Gamma>\dim N-\dim U=d-1-\dim U$, 
the Burger-Roblin measure $\lambda$ is $U$-conservative. 

In the case $\delta_\Gamma<\dim N-\dim U$,   we prove below (Theorem \ref{divergence}) that the measure
$\lambda$ is $U$-dissipative.  
Unfortunately, our method  does not work   in the case $\delta_\Gamma=\dim N-\dim U$. 
We refer to works of Dufloux \cite{Dufloux2016} and \cite{Dufloux2017}
for the proof that 
\begin{itemize}
\item When $\mu$ is finite and $\Gamma$ Zariski dense, the measure $\mu$ is $U$-dissipative iff $\delta_\Gamma\le \dim N-\dim U$\\
\item When moreover $\Gamma$ is convex-cocompact, if $\delta_\Gamma=\dim N-\dim U$, then $\lambda$ is $U$-conservative. 
\end{itemize}


\begin{theo}\label{divergence} Let $\Gamma$ be a discrete Zariski dense subgroup of $G={\SO}_o(d,1)$  group and $U<G$ a unipotent subgroup. If $\delta<d-1-\dim U$, then for all compact sets 
$K\subset \F\M$ and $\lambda$-almost all {\rm $\x\in \F M$} the time spent by {\rm$ \x  U$} in $K$ is finite. 
\end{theo}


 Let $d=\dim  U$.
 Let $r>0$. Let $N_r\subset N$ (resp. $U_r\subset U$) be the closed ball of radius $r>0$ and center $0 $ in $ N$ (resp. in $U$). 
Let $K_r=K.N_r $ be the $r$-neighbourhood
of $K $ along the $N$-direction.  


Let $\mu_{\x N}$ be the conditional measure on $W^{ss} (\x)=\x N$ 
of  the Bowen-Margulis measure. 
 
\begin{figure}[ht!]\label{dissipative }
\begin{center}
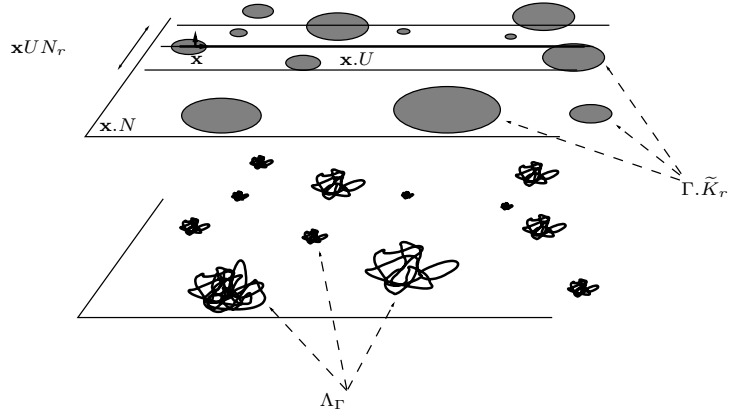
\caption{Intersection of a $U$-orbit with the $\Gamma$-orbit of a compact set $\widetilde{K}_r$} 
\end{center}
\end{figure}

\begin{lem}\label{comparison-noncompact}\rm For all 
compact sets  $K\subset \Omega$, and all $\x\in \mathcal{E}$, if
$K_r=K.N_r$, for all $r>0$, there exists $c=c(x,r,K)>0$ such that 
\[
\int_ U 1_{K_{2r}}(\x u ) du \leq c \, \mu_{\x N }(\x U N_{2r} ).
\]
\end{lem}

\begin{proof} 
Let $r>0$. First, the map $\x\in\Omega \mapsto \mu_{ \x N  }(\x N_r )$ is continuous. It is
 an immediate consequence of \cite[Cor. 1.4]{FS}, where they establish that $\mu_{\x N}(\partial\x N_r)=0$ for all
$\x\in \Gamma\backslash G$ and $r>0$. 
 In this reference, they assume at the beginning $\Gamma$ to be convex-cocompact, 
but they use in the proof of corollary 1.4 only the finiteness of $\mu$. 

The above map is also  positive, and therefore 
bounded away from $0$ and $+\infty$ on any compact set. 
 Let $0<c_r=\inf_{\z \in K_{4r} } \mu_{\z N   }(\z N_r )\le C_r=\sup_{\z \in K_{4r} } \mu_{ \z N   }(\z N_r )<\infty$. 

Let us work now on $G$ and not on $\Gamma\backslash G$. 
Fix a frame $\x\in \widetilde{\mathcal{E}}\subset \F \H^d$.  
For all $\y\in \x U\cap \Gamma K N_r$,  choose some  $\z \in \y N_{r}\cap \Gamma K$ 
and consider the ball $\z N_r$. 
Choose among them  a maximal (countable) family of balls $\z_i N_{r}\subset \x U N_{2r}$ which are pairwise disjoint. By maximality, the family of balls $z_i N_{4r}$ cover $\x U N_{r}\cap \Gamma K N_r$.  
 
We deduce on the one hand 
\[
\int_\U 1_{K_r}(\x u) du \leq\sum_i \mu_{\x N} (\z_i N_{4r}) \le C_{4r} |I|\,.
\]
On the other hand, 
as the balls $\z_i N_{r}$ are disjoint, 
\[
\mu_{\x N}(\x U N_{2r})\ge \sum_i \mu_{\x N} (\z_i N_r)\ge c_r |I|. 
\]
This proves the  lemma.
\end{proof}


 To prove  Theorem \ref{divergence}, it is therefore sufficient to prove the following lemma. 
 
\begin{lem}\label{integrale-finie} Assume that  $\delta_\Gamma<\dim N-\dim U$. Then  for all {\rm $\x \in \mathcal{E}$} such that
{\rm $\frac{d(\x, \x a_t)}{t}\to 0$} when $t\to +\infty$, we have \rm 
\[
\int_{M} \mu_{\x m N}(\x m  U N_r) dm<\infty.
\]
\end{lem}

Indeed, Lemma \ref{dist} ensures that the assumption of  Lemma \ref{integrale-finie}
 is satisfied $\lambda$-almost surely. 
And by Lemma \ref{comparison-noncompact}, its conclusion implies 
that for $\lambda$-a.e. $\x\in\mathcal{E}$ and almost all $m\in M$, the 
orbit $\x m U$ does not return infinitely often in a compact set $K$. 
As $\lambda$ is by construction the lift to $\F M$ of $m_{BR}$ on $T^1M$, with the Haar measure of $M$ on the fibers, this implies that for $\lambda$-almost all $\x$, the orbit $\x m U$ does not return infinitely often in a compact set $K$. This implies the dissipativity of $\lambda$ w.r.t. the action of $U$,
so that Theorem \ref{divergence} is proved. 

\begin{proof} Recall first that for $n\in N$ not too small, 
one has $d\mu_{\x N}(\x n)\simeq |n|^{2\delta} d\nu((\x n)^-)$. 
We want to estimate the   integral $\int_{M} \mu_{\x m N }(\x m  U N_r  ) d m$. 

First, observe that the measure $\mu_{\x N}$ on $\x N$ does not depend really on the orbit $\x N$, in
the sense that it is the lift of a measure on $W^{ss}(\pi_1(\x))$ through the inverse of the canonical projection $\y\in\x N\to \pi_1(\y)$ from $\x N$ to $W^{ss}(\pi_1(\x))$.  
Therefore, one has $\mu_{\x m N}(  \x m  U N_r)=\mu_{\x N}(\x m U m^{-1} N_r)$. 

Thus, by Fubini Theorem, one can compute\,: 
\begin{align*}
F(\x)& =\int_{M} \mu_{\x m N }(\x m  U N_r  ) d m \\
& =\int_{M} \mu_{\x N}(\x m  U m^{-1} N_r ) d m \\
& =\int_{M\times N} 1_{m \in M, m U m^{-1}\cap n N_r\neq\emptyset }(m)d m d\mu_{\x N}(n)\\
& \simeq \int_{N_0} r^{\dim N-\dim U} |n|^{\dim U-\dim N} d\mu_{\x N}(n)\\
& \simeq \int_{N_0} r^{\dim N-\dim U} |n|^{ \dim U-\dim N+2\delta} d\nu ((\x n)^-) \\
\end{align*}
 where $N_t=\{n\in N; |n|\geq 2^t\}$. 

The estimate comes from the probability that a point in the sphere of dimension $k-2$ falls in the $r/|n|$-neigborhood of a fixed subsphere of dimension $d-1$, see for example \cite[chapter 3]{MR1333890}.

  Therefore,  we get 
\[F(\x) 
\le \sum_{l\ge 0} 2^{l(\dim U-\dim N+2\delta)} \nu((\x N_l)^-)
\]

Now, observe that $(\x N_l)^-$ is comparable to the ball of center $x^+$ and radius $2^{-l}$ on the
boundary. 
By  Inequality (\ref{estimee-nu}), we deduce that
$$\nu((\x N_l)^-)\le 2^{-\delta l}e^{\delta_\Gamma d(\x a_{l\log 2},\Gamma o)}$$

For all $\varepsilon >0$, there exists $l_0\ge 0$, such that $d(\x a_{l\log 2})\leq  \varepsilon l\log 2$ for $l\ge l_0$. 
Thus, up to the $l_0$ first terms of the series, we  get the following upper bound for $F(\x)$. 
\begin{eqnarray*}
F(\x) 
& \le &  \sum_{l=0}^{l_0-1}\dots +\sum_{l\ge 0} 2^{l(\dim U-\dim N+\delta)}e^{\delta_\Gamma d(\x a_{l\log 2},\Gamma o)}\\
&\le &\sum_{l=0}^{l_0-1}\dots +\sum_{l\ge l_0} 2^{l(\dim U-\dim N+\delta+\varepsilon\delta)}
\end{eqnarray*}

Thus, if $\delta<\dim N-\dim U$, we can choose $\varepsilon>0$  
so that $\dim U-\dim N+\delta+\varepsilon\delta<0$, and $F(\x)$ is finite. 
\end{proof}

\begin{rema}\rm
Observe that the above argument, in the case $\delta+\dim  U= \dim N$, would   lead to the fact 
that 
\[
\int_{M} \mu_{\x m N}(\x m  U N_r) dm=\infty, 
\]
which is not enough to conclude to the conservativity, that is that almost surely, $\mu_{\x m N}(\x m  U N_r)=+\infty$. 
We refer to the works of Dufloux for a finer analysis in this case. 
\end{rema}


 \subsection{Equivalence of the Bowen-Margulis-Sullivan
 measure and the Burger-Roblin measure for invariants sets}

 As claimed in the introduction, we reduce the study of 
ergodicity of the Burger-Roblin measure $\lambda$ to the ergodicity of 
the Bowen-Margulis-Sullivan measure $\mu$. The rest of the section is devoted to the proof of the following Proposition:

\begin{prop} \label{BMandBRequivalent} Assume that $\Gamma$ is Zariski-dense.
If $\mu$ finite and $\delta_\Gamma +\dim(U)>d-1$, 
then for any $U$-invariant Borel set $E$, we have $\lambda(E)>0$ if and only if $\mu(E)>0$.
\end{prop}

 We denote by $\mathcal{B}$ the Borel $\sigma$-algebra of $\mathcal{E}$, and $\mathcal{I}_U \subset \mathcal{B}$ 
the sub-$\sigma$-algebra of $U$-invariant sets. 
The first part of the proof of Proposition \ref{BMandBRequivalent} is the following.
 
 \begin{lem} \label{muposimplieslamdapos}
 Assume that $\Gamma$ is Zarisi-dense in $\SO_o(d,1)$ and that $\mu$ is finite. 
If $\delta>\dim N-\dim U$ and $E$ is a Borel $U$-invariant set such that $\mu(E)>0$, then $\lambda(E)>0$.
 \end{lem}
 \begin{proof} 

Let $E$ be a Borel $U$-invariant set with $\mu(E)>0$.
 It is sufficient to show that $\tilde{\lambda}(\tilde{E})>0$. Let $\x_0=(\x_0^+,x_0^-,t_{\x_0})$ be a frame in the support of the (non-zero) measure $1_{\tilde{E}}\tilde{\mu}$, and $F$ be a small neighbourhood of $\x_0$. Denote by  $\mathcal{H}(x^+,t_\x)$ the horosphere passing through the base-point of the frame $\x$. 
The measure $\tilde{\mu}(\tilde{E}\cap F)$ can be written
 $$
\tilde{\mu}(\tilde{E}\cap F)= 
\int_{\mathcal{F}\Lambda_\Gamma \times \R} \left( \int_{\mathcal{H}(x^+,t)}  1_{\tilde{E}\cap F} (\x^+,x^-,t)\,.g 
 d\nu(x^-) \right) d\tilde{\nu}(\x^+)dt_\x,
$$
 where $g$ is a positive continuous function, namely the exponential of some Busemann functions, and $\tilde{\nu}$ the $M$-invariant lift of $\nu$ to $\mathcal{F}\Lambda_\Gamma$. The main point is that it is positive, so for a set $J\subset \mathcal{F}\Lambda_\Gamma \times \R$ of positive $\tilde{\nu}\otimes dt$ measure,  for any $(\x^+,t_\x)\in J$, the set
 $$E_{\x^+,t}^F=\{x^-  \; :\; (\x^+,x^-,t_\x)\in \tilde{E}\cap F, \},$$
 has positive $\nu$-measure. \\
 
\begin{figure}[ht!]\label{comparison-mu-lambda}
\begin{center}
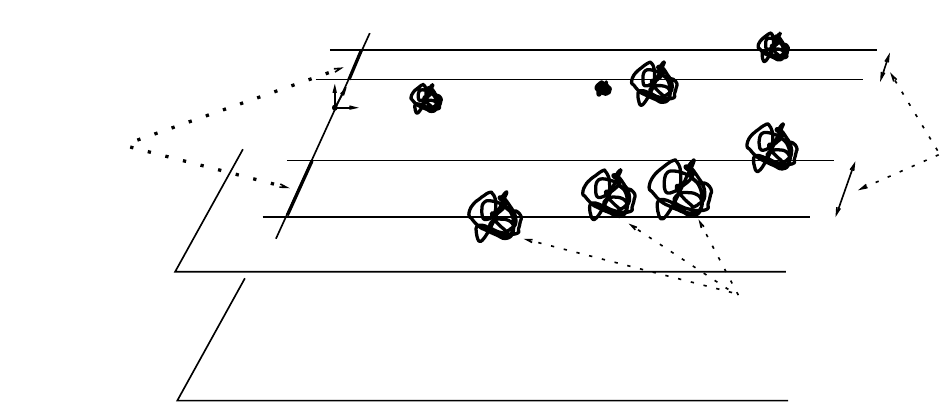
\caption{  } 
\end{center}
\end{figure}  
 
 Since similarly,
 $$\tilde{\mu}(\tilde{E})= \int_{\mathcal{F}\Lambda_\Gamma \times \R} \left( \int_{\mathcal{H}(x^+,t_\x)} g'.1_{\tilde{E}} (\x^+,x^-,t_\x)  d\mathcal{L}(x^-) \right) d\tilde{\nu}(\x^+)dt_\x,$$
with $g'>0$, it is sufficient to show that for a subset of $(\x^+,t_\x) \in J$ of positive measure, the set 

 $$
E_{\x^+,t}=\{x^-  \; :\; (\x^+,x^-,t_\x)\in \tilde{E}, \}
$$
has positive Lebesgue $\mathcal{L}$-measure.\\

On each horosphere $\mathcal{H}(x^+,t_\x)$, we wish to use Marstrand's projection Theorem, and therefore to use 
an identification of the horosphere with $\R^{d-1}$. A naive way would be to say that 
$\mathcal{H}(x^+,t_\x)$ is diffeomorphic to $\x N$, and therefore to $N\simeq \R^{d-1}$. 
However,    it will be more convenient to use an identification of these horospheres with $N\simeq \R^{d-1}$ which does not depend on a frame $\x$ in $\pi_1^{-1}(\mathcal{H}(x^+,t_\x))$. 

 In order to obtain these convenient coordinates, we fix a smooth section $s$ 
from a neighbourhood of $x_0^+$ to $\mathcal{F} \partial \H^d$. If $\x \in F$, the horosphere $\mathcal{H}(x^+,t_\x)$  can be identified (in a non-canonical way) with $N$ the following way: let $n\in N$, we associate to it the base-point of $(s(x^+),x_0^-,t_\x)n$. This way, the identification does depend only on the $MN$-orbit of $\x$, that is depends on the horosphere only.\\
 
 For $\x^+\in \mathcal{F}\Lambda_\Gamma$, define $m=m(\x^+) \in M$ by the relation $\x^+=s(x^+)m$. 
If $\x \in \tilde{E}$, then so does $\x u= (s(x^+)m,x_0^-,t_\x)u
 =(s(x^+),x_0^-,t_\x)mu$, which has the same base-point as $(s(x^+),x_0^-,t_\x)mum^{-1}$. This means that the set $E_{\x^+,t}$, viewed as a subset of $N$, is invariant by translations by the subspace $mUm^{-1}$ in these coordinates. From now on, $E_{\x^+,t}$ will always be seen as a subset of $N$. \\ 
 
 Let $V$ be the orthogonal complement of $U$ in $N$, and $\Pi_{mVm^{-1}}:N\to mVm^{-1}$ be the orthogonal projection onto $mVm^{-1}$. What we saw is that the set $E_{\x^+,t}$ is a product of $mUm^{-1}$ and $\Pi_{mVm^{-1}}(E_{\x^+,t})$. Clearly, it contains the product of $mUm^{-1}$ and $\Pi_{mVm^{-1}}(E_{\x^+,t}^F)$, so it is of positive Lebesgue $V$-measure as soon as $\Pi_{mVm^{-1}}(E_{\x^+,t}^F)$ has positive Lebesgue measure in $mVm^{-1}$.
 
 The strategy is now to use the projection Theorem \ref{projL2} on each horosphere to deduce that $\Pi_{mVm^{-1}}(E_{\x^+,t}^F)$ is of positive Lebesgue measure for almost every $m \in M$. Unfortunately, we cannot apply it to the measure $1_{E_{\x^+,t}^F}\nu$ directly, since the set $E_{\x^+,t}^F$ depends on the orientation $m$ of the frame $\x^+=s(x^+)m$ (and not only on the Horosphere $\mathcal{H}(x^+,t_\x)$), so it depends on $M$. \\
 
 By Lemma \ref{energie-finie}, we can find a subset $L\subset \Lambda_\Gamma$, such that  $\nu_{|L}$ has finite $\dim(N)-\dim(U)$-energy, and  $E_{\x^+,t}^F\cap L$ has positive $\nu$-measure for any $(\x^+,t)\in J'$, where $J'\subset J$ is of positive $\tilde{\nu}\otimes dt$- measure. \\
 
 One can moreover assume that for every horosphere $\mathcal{H}(x^+,t_\x)$ with $\x \in F$, $L$ lies in a fixed compact set of $N$ using both identifications of the horosphere with $\partial \H^d$ and $N$. Notice that these identifications are smooth maps, so the finiteness of the energy of $\nu_{|L}$ does not depend on the model metric space chosen.\\
 
 By Theorem \ref{projL2}, applied on each horosphere $\mathcal{H}(x^+,t)\simeq N=U\oplus V$, the orthogonal projection $(\Pi_{mVm^{-1}})_*\nu_{|L}$ is $m$-almost surely absolutely continuous with respect to the Lebesgue measure  on $mVm^{-1}$. But since $1_{E_{\x^+,t}^F}\nu_{|L} \ll \nu_{|L}$, we have for almost every $m$
 $$(\Pi_{mVm^{-1}})_* (1_{E_{\x^+,t}^F}\nu_{|L}) \ll (\Pi_{mVm^{-1}})_*\nu_{|L} \ll \mathcal{L}_{mVm^{-1}}.$$
 This forces the projection set $\Pi_{mVm^{-1}} E_{\x^+,t}^F$ to be of positive $\mathcal{L}_{mVm^{-1}}$-measure $m$-almost surely, for those $m$ such that 
 $(s(x^+)m,t)\in J'$.
 \end{proof}

The second step of the proof is the following. 

\begin{lem} Assume that $\Gamma$ is Zariski-dense in ${SO^o(d,1)}$, that $\mu$ is ergodic and conservative, 
and $\delta>\dim N-\dim U$. 
If $E$ is a Borel $U$-invariant set such that $\mu(E)=1$, then $\lambda(E\triangle \mathcal{E})=0$.  
\end{lem} 

\begin{proof} First, pick some  element $a\in A$ whose adjoint action has eigenvalue $\log(\lambda_a)>0$ on $\mathfrak{n}$  such that 
$a n a^{-1} =\lambda_a v$ for all $n \in N$.

Replacing $E$ by $\cap_{k\in \Z}E.a^k$ (another set of full $\mu$-measure), 
we can freely assume that $E$ is $a$-invariant. \\

By Lemma \ref{muposimplieslamdapos}, we already know that $\lambda(E)>0$. As above also, let $V\subset N$ be a
supplementary of $U$ in $N$. 
As $\lambda(E)>0$, we know that for $\lambda$-almost all $\x\in E$, the set 
$V(\x,E)=\{v\in V, \x v\in E\}$ has positive $V$-Lebesgue (Haar) measure $dv$.

The Lebesgue density points of $V(\x,E)$  have full $dv$-measure. 
Recall that $V_t$ is the ball of radius $t$ in $V$. 

Let $\epsilon \in (0,1)$, and define for all $\x\in \mathcal{E}$ (not only $\x\in E$)
\[
F_{\varepsilon,E}(\x)=\sup\left\{T>0 \, : \, \forall t \in (0,T), \int_V {\bf 1}_{\x V_t \cap E} dv \ge (1-\varepsilon) |V_t|\right\}\,,
\]
with the convention that it is zero if no such $T$ exists; it may take the value $+\infty$. Observe that $F_{\varepsilon,E}$ is a $U$-invariant map, because $E$ is $U$-invariant. \\

Since the Lebesgue density points of $V(\x,E)$ have full $dv$-measure, then for $\lambda$-almost all $\x \in E$, and $dv$-almost all $v\in V(\x,E)$, $F_{\varepsilon,E}(\x v)>0$. Moreover, this statement stay valid for other $U$-invariant sets $E'$ of positive $\lambda$-measure.\\


 We claim that for $\mu$-almost every $\x \in E$, $F_{\varepsilon,E}(\x)>0$. 
Assuming the contrary, $E'=F_{\varepsilon,E}^{-1}(0)\cap E$ is a $U$-invariant 
 set of positive $\mu$-measure, so by Lemma \ref{muposimplieslamdapos}, it 
is also of positive $\lambda$-measure. 
As   $E'\subset E$, $F_{\varepsilon,E'}\leq F_{\varepsilon,E}$, so that the function $F_{\varepsilon, E'}$ 
is identically zero on $E'$. 
But there exists $\x \in E'$ and $v\in V(\x,E')$ such that $\x v\in E'$ (by definition of $V(\x,E')$)  and $F_{\varepsilon,E'}(\x v)>0$, by the previous consideration of Lebesgue density points, leading to an absurdity.\\

We will now show that $F_{\varepsilon,E}$ is in fact infinite, $\mu$-almost surely. 
First, the classical commutation relations between $A$ and $N$ (and therefore $A$ and $V\subset N$) give
 $a V_T a^{-1}=V_{\lambda_a T}$.   Observe also that,by $a$-invariance of $E$,
\[
V(\x a, E)=\{v,\x a v\in E\}=\{v\in V, \x a v a^{-1}=\x.(\lambda_a.v)\in E \}= \lambda_a^{-1} V(\x,E).
\]
Therefore, $F_{\varepsilon,E}(\x a)= \lambda_a  F_{\varepsilon,E}(\x)$, 
i.e. it is a function increasing along the dynamic of an ergodic and conservative measure-preserving system. 
This situation is constrained by the conservativity of $\mu$. Indeed, 
assume there exists $t_1<t_2$ such that $\mu(F_{\varepsilon,E}^{-1}(t_1,t_2))>0$. 
Then for all $k$ large enough (namely s.t. $\lambda_a^k>t_2/t_1$), we have
$$
\left(F_{\varepsilon,E}^{-1}(t_1,t_2) \right) a^k\cap \left(F_{\varepsilon,E}^{-1}(t_1,t_2) \right)=\emptyset,
$$
in contradiction to the conservativity of $\mu$ w.r.t. the action of $a$.\\
 
This shows that $F_{\varepsilon,E}(\x)=+\infty$ for $\mu$-almost all $\x \in \mathcal{E}$. 

Define now $\mathcal{I}_E=\cap_{j\in\N^*} F_{1/j,E}^{-1}(+\infty)$.
It is a $U$-invariant set of full $\mu$-measure as a countable intersection of sets of full $\mu$-measure. 
Therefore $\lambda(\mathcal{I})>0$ by Lemma \ref{muposimplieslamdapos}. 
By definition of $F_{\varepsilon, E}$, $\mathcal{I}$ consists of the frames $\x$ such that $V(\x,E)$ is of full measure in $V$, a property that is $V$-invariant. Hence $\mathcal{I}$ is $N$-invariant of positive $\lambda$-measure, so by ergodicity of $(N,\lambda)$, it is of full $\lambda$-measure.\\

 Unfortunately, we know that $E\subset \mathcal{I}_E$ but $\mathcal{I}_E$ does not have to be a subset of $E$. 
To be able to conclude the proof (i.e. show that $\lambda(E^c)=0$),  we consider the complement set $E'=E^c$, and assume it to be of positive $\lambda$-measure. For any $\x \in \mathcal{I}_E$ and $v\in V$, by definition of $\mathcal{I}_E$, $F_{\varepsilon,E^c}(\x v)=0$. So the intersection of $\mathcal{I}_E$ and $E^c$ is of zero measure, and thus $\lambda(E^c)=0$.

 
\end{proof}

Let us now conclude the proof of Proposition \ref{BMandBRequivalent}.   
 Let $E$ be a $U$-invariant set. We already know that $\mu(E)>0$ implies $\lambda(E)>0$. 
For the other direction,  assume that $\mu(E)=0$, so that
$\mu(E^c)=1$. The above Lemma applied to $E^c$ therefore would imply $E^c=\mathcal{E}$
 $\lambda$-almost surely, so that $\lambda(E)=0$. 
Thus, $\lambda(E)>0$ implies  $\mu(E)>0$.


\section{Ergodicity of the Bowen-Margulis-Sullivan measure}


 \subsection{Typical couples for the negative geodesic flow}
 
  Let us say that a couple $(\x,\y)\in \Omega^2$ is {\em typical} (for $\mu\otimes \mu$) 
if for {\em every} compactly supported continuous 
 function $f\in C^0_c(\mathcal{E}^2)$, the conclusion of the Birkhoff ergodic Theorem holds for the couple $(\x,\y)$ in negative discrete time for the action of $a_1$, more precisely:
  \[
   \lim_{N\rightarrow +\infty} \frac1{N} \sum_{k=0}^{N-1} f(\x a_{-k},\y a_{-k}) = \mu\otimes \mu(f).
  \]
 Write $\mathcal{T}$ for the set of typical couples, which is a subset of the set of generic couples. 

 \medskip
 
Let us explain briefly why this is a set of full $\mu\otimes \mu$-measure. 
Since the action of $A$ on $(\Omega,\mu)$ is mixing, so is the action of $a_{-1}$.
 A fortiori, the action of $a_{-1}$ on $(\Omega,\mu)$ is weak-mixing, so the diagonal action of $a_{-1}$ on $(\Omega^2,\mu\otimes \mu)$ is ergodic. 
It follows from the Birkhoff ergodic Theorem applied to a countable dense subset of the separable space 
$(C^0_c(\mathcal{E}^2),\|.\|_\infty)$ that $\mu\otimes\mu$-almost every couple is typical. 
 
 \medskip
 
As the set of generic couples used in the topological part of the article (see section \ref{sectiontop}), 
the subset of typical couples enjoys the same nice invariance properties by $\left( (M\times A)\ltimes N^-\right)^2$.
 That is, $(\x,\y)\in (\F \H^d)^2$ being the lift of a typical couple only depends on $(x^-,y^-)$ in Hopf coordinates. 
This follows from the fact that $M\times A$ acts isometrically on $C_c^0(\mathcal{E}^2)$ and commutes with $a_{-1}$, 
so $\mathcal{T}$ is $(M\times A)^2$-invariant, and the fact that, 
since elements of $C^0_c(\mathcal{E}^2)$ are uniformly continuous,
 two orbits in the same strong unstable leaf have the same limit for their ergodic averages.


 \subsection{Plenty of typical couples on the same $U$-orbit}
\label{subsectionplenty}

 We will say that there are {\em plenty of typical couples on the same $U$-orbit} if there exists a probability measure $\eta$ on $\Omega^2$ 
such that the three following conditions are satisfied:
 \begin{enumerate}\label{plenty}
 \item Typical couples are of full $\eta$-measure, that is $\eta(\mathcal{T})=1$.
 \item Let $ p_1(\x,\y)=\x,  p_2(\x,\y)=\y$ be the coordinates projections. 
We assume that, for $i=1,2$, $( p_i)_*\eta$ is absolutely continuous with respect to $\mu$. 
We denote by $D_1$, $D_2$ their respective Radon-Nikodym derivatives, so that $( p_i)_*\eta=D_i \mu$.
  We assume moreover that $D_2 \in L^{2}(\mu)$. 
 \item Let $\eta_\x$ and $\eta^\y$ be the measures on $\Omega$ obtained by disintegration of $\eta$ along the maps $ p_i$, $i=1,2$ respectively. 
More precisely, for any $f\in L^1(\eta)$,
 \[
  \int_{\Omega^2} f d\eta= \int_\Omega \left( \int_\Omega f(\x,\y) d\eta_\x(\y) \right) d\mu(\x)= \int_\Omega \left( \int_\Omega f(\x,\y) d\eta^\y(\x) \right) d\mu(\y).
 \]
 Note that $\eta_\x$ (resp $\eta^\y$) have total mass $D_1(\x)$ (resp. $D_2(\y)$). 
Whenever this makes sense, define the operator $\Phi$ which to a function $f$ on $\Omega$ associates the following function on $\Omega$:
\[
\Phi(f)(\x)=\int_{\Omega} f(\y)d\eta_\x(\y).
\]
 The condition (3) here is that if $f$ is a bounded, measurable $U$-invariant function, then 
 \[
 \Phi(f)(\x)=f(\x)D_1(\x)
 \]
  for $\mu$-almost every $\x\in \Omega$. Note that even if $f$ is bounded, $\Phi(f)$ may not be defined everywhere.
 \end{enumerate}
 
\begin{rema}\rm Observe that we do not require any invariance of the measure $\eta$. Condition (1) replaces the $A$-invariance, 
whereas Condition (3) establish a link between the structure of $U$-orbits and $\eta$. 
\end{rema}

\begin{rema}\label{eta-supported-on-U-orbits}\rm 
 Let us comment a little bit on condition (3): it is obviously satisfied if, for example, $\eta_\x$ is supported on $\x U$ for almost every $\x$, 
that is, $\eta$ is supported on couples of the form $(\x,\x u)$ with $u\in U$. 
It will be the case for the measures $\eta$ we will construct in section \ref{constrplentydim3} and \ref{higherdimeta} in dimension 3 and higher respectively. \\
 A good example of a measure $\eta$ satisfying (2) and (3) is the following: let $(\mu_\x)_{\x\in \Omega}$ 
be the conditional measures of $\mu$ with respect to the $\sigma$-algebra of $U$-invariant sets, 
and define $\eta$ as the measure on $\Omega^2$ such that $\eta_\x=\mu_\x$ by the above disintegration along $ p_1$. 
However, its seems difficult to prove directly that it also satisfies (1). 
This example also highlights that condition (3) is in fact weaker than requiring that $\eta_\x$ is supported on $\x U$.
\end{rema}

\begin{rema}\label{pas-besoin-de-L2}\rm The condition that the Radon-Nikodym derivatives $D_i$ be in $L^2$ is not restrictive. 
Indeed , we will 
construct a measure $\eta'$ satisfying all above conditions except this $L^2$-condition. 
The Radon-Nikodym derivatives $D_i$ are 
integrable, so that they are bounded on a set of large measure. 
We will simply restrict $\eta'$ 
to this subset, and normalize it, 
to get the desired probability measure $\eta$.  
\end{rema}
 \medskip

 The interest we have in finding plenty of typical couples on the same $U$-orbit is due to the following key observation.
 \medskip
 
 \begin{lem} \label{reductiontwo}
  {\em To prove Theorem \ref{ergodicity}, it is sufficient to prove that there are plenty of typical couples on the same
$U$-orbit, that is that there exists  a probability measure $\eta$ satisfying (1),(2) and (3).}
 \end{lem} 
\medskip
  
 The next section is devoted to the proof of this observation. 
The idea is the following: suppose $g$ is a bounded $U$-invariant function. We aim to prove that $g$ is constant $\mu$-almost everywhere.
Consider the integral of the ergodic averages for the function $g\otimes g$ on $\Omega^2$ with respect to $\eta$,

$$ J_N=\int_{\Omega^2} \frac1{N} \sum_{k=0}^{N-1} g\otimes g(\x a_{-k} , \y a_{-k}) d\eta(\x,\y).$$
If $\eta$ is supported only on couples on the same $U$-orbit, then since $g$ is constant on $U$-orbits, $g(\x a_{-k})=g( \y a_{-k})$
 for $\eta$-almost every $(\x,\y)$, so
\begin{align*}
 J_N & = \int_{\Omega^2} \frac1{N} \sum_{k=0}^{N-1}  g(\x a_{-k})^2 d\eta(\x,\y)\\
 &=\int_\Omega \frac1{N} \sum_{k=0}^{N-1} g(\x a_{-k})^2 D_1(\x) d\mu(\x),
 &=\int_\Omega g(\x)^2 \left( \frac1{N} \sum_{k=0}^{N-1}  D_1(\x a _k)  \right) d\mu(\x),
\end{align*}
so  $J_N\to \int_\Omega g^2 d\mu$ by the Birkhoff ergodic Theorem applied to $D_1$.
Observe that Property (3) is used in the first equality, and Property (2) in the second.

 For the sake of the argument, assume that $g$ is moreover {\em continuous with compact support}. Then by Condition (1) on typical couples, since $g\otimes g$ is continuous with compact support, the same sequence $J_N$ tends to $\int_{\Omega^2} g\otimes g d\mu=(\int_\Omega g d\mu)^2.$
Hence $g$ has zero variance, so is constant. Unfortunately, one cannot assume $g$ to be continuous, 
nor approximate it by continuous functions in $L^\infty(\mu)$. The regularity Condition (2) that $D_2 \in L^{2}$ will nevertheless 
allow us to use continuous approximations in $L^2(\mu)$.


 \subsection{Proof of Lemma \ref{reductiontwo}  }
  
 We first need to collect some facts about the operator $\Phi$, 
and its behaviour in relationship with ergodic averages for the negative-time geodesic flow $a_{-1}$. 
    
 \begin{lem} The operator $\Phi$ is a continuous linear operator from $L^2(\mu)$ to $L^1(\mu)$.
 \end{lem}
As we will see, Property (2) of the measure $\eta$ is crucial here. 
 \begin{proof} Let $f\in L^2(\mu)$, we compute
 \begin{align*}
  \|\Phi(f)\|_{L^1(\mu)} & = \int_\Omega \left| \Phi(f)(\x) \right| d\mu(\x)\leq \int_\Omega \left( \int_\Omega |f(\y)| d\eta_\x(\y) \right) d\mu(\x), \\
  & \leq \int_{\Omega^2} |f(\y)| d\eta(\x,\y) \leq \int_\Omega |f(\y)| \left( \int_\Omega  d\eta^\y(\x) \right) d\mu(\y),\\
  & \leq \int_\Omega |f(\y)| D_2(\y) d\mu(\y)\leq \|f\|_{L^2(\mu)} \, \|D_2\|_{L^{2}(\mu)}.
 \end{align*}
 \end{proof}
 
 Given $f,g$ two functions on $\Omega$, write $f\otimes g$ for the function $f\otimes g(\x,\y)=f(\x)g(\y)$ on $\Omega^2$.
Denote by $\langle f,g \rangle_\mu=\int_\Omega f.g\,d\mu$ the usual scalar product on $L^2(\mu)$.  For $f \in L^\infty(\mu)$ and $g\in L^2(\mu)$, a simple calculation gives 
\[
\int_{\Omega^2}f\otimes g \, d\eta= \langle f, \Phi(g)\rangle_\mu. 
\]
 
 Let $\Psi$ be the Koopman operator associated to $a_{1}$, that is $\Psi(f)(\x)=f(\x a_1 )$. 

The Ergodic average of a tensor product can be written in terms of $\Phi$ and $\Psi$ the following way:
 \begin{align*}
  \int_{\Omega^2} \frac1{N} \sum_{k=0}^{N-1} f\otimes g(\x a_{-k} , \y a_{-k}) d\eta(\x,\y) &= 
\frac1{N} \sum_{k=0}^{N-1} \langle \Psi^{-k}(f), \Phi( \Psi^{-k}(g))\rangle_\mu,\\
  & = \langle f, \frac1{N} \sum_{k=0}^{N-1} \Psi^k \circ \Phi \circ \Psi^{-k}(g) \rangle_\mu \\
  & = \langle f, \Xi_N(g) \rangle_\mu,
  \end{align*}
 where $\Xi_N$ is the operator $\Xi_N =\frac1{N}\sum_{k=0}^{N-1} \Psi^k \circ \Phi \circ \Psi^{-k}$. 
Since the Koopman operator is an isometry from $L^q(\mu)$ to $L^q(\mu)$ for both $q=1$ and $q=2$, 
the operator $\Xi_N$ from $L^2(\mu)$ to $L^1(\mu)$ has norm at most 

$$\|\Xi_N\|_{L^2 \to L^1}\leq \|\Phi\|_{L^2 \to L^1}.$$

 
  Notice also that if $f,g$ are continuous with compact support, the above ergodic average converges toward 
$\langle f,1\rangle_\mu \langle g,1\rangle_\mu$ for $\eta$-almost every $x,y$, by Condition (1). 
By the Lebesgue dominated convergence Theorem, we also have
  \begin{equation}\label{Xi_n-converge}
   \lim_{N\to \infty} \langle f,\Xi_N(g)\rangle_\mu =\langle f,1\rangle_\mu \langle g,1\rangle_\mu.
  \end{equation}
 
 \medskip
 
 Let $g$ be a  bounded measurable, $U$-invariant function. Since $\Psi^{-k}(g)$ 
is also bounded and $U$-invariant, by property (3), we have
 \[
  \Phi(\Psi^{-k}(g))(\x)=g(\x a_{-k})D_1(\x).
 \]
 Therefore, 
 \[
 \Xi_N(g)(\x)=g(\x)\left(\frac1{N}\sum_{k=0}^{N-1} D_1(\x a_k  )\right).
 \]
  By the Birkhoff $L^1$-ergodic Theorem and boundedness of $g$, 
it follows that $\Xi_N(g)$ tends to $g$ in $L^1(\mu)$-topology.
  
  \medskip
  
 Our aim is to show that $g$ has variance zero. Let $(g_n)_{n\geq 0}$ be a sequence of 
uniformly bounded continuous functions with compact support converging to
 $g$ in $L^2(\mu)$ (and hence also in $L^1(\mu)$). 
Let $D>0$ be such that $\|g_n\|_\infty\leq D$ for all $n$. 
For $n,N$ positive integers, we have
 \begin{align*}
  \langle g,g \rangle_\mu - \langle g,1\rangle_\mu^2 = & \, \langle g-g_n,g \rangle_\mu + \langle g_n, g-\Xi_N(g) \rangle_\mu + \langle g_n, \Xi_N(g-g_n) \rangle_\mu \\
  & \,    + \left( \langle g_n, \Xi_N(g_n) \rangle_\mu-\langle g_n,1\rangle_\mu^2  \right)+ 
\left( \langle g_n,1\rangle_\mu^2 - \langle g,1\rangle_\mu^2\right).
 \end{align*} 
 
 Therefore,
 
\begin{align*}
 \left|  \langle g,g \rangle_\mu -\langle g, 1\rangle_\mu^2 \right| \leq & \|g-g_n\|_1\|g\|_\infty + D \|g-\Xi_N(g)\|_1 
 + D \|\Xi_N\|_{L^2\to L^1} \|g-g_n\|_2 \\
 &  + \left| \langle g_n, \Xi_N(g_n) \rangle_\mu-\langle g_n,1\rangle_\mu^2  \right|+ \|g-g_n\|_1\|g+g_n\|_1.
\end{align*}
 
 First fix $n$ and let $N$ goes to infinity. By what precedes, $\Xi_N(g)$ converges to $g$ in $L^1$ so that the
second term vanishes.  Since $g_n$ is continuous, by  (\ref{Xi_n-converge}),
 the last but one term of the upper bound vanishes. We obtain

\[
 \left|  \langle g,g \rangle_\mu -\langle g, 1\rangle_\mu^2 \right| \leq  \|g-g_n\|_1\|g\|_\infty  + D \|\Phi\|_{L^2\to L^1} \|g-g_n\|_2  + 2D \|g-g_n\|_1 .
\]
 We now let $n$ go to infinity, and we get
\[
\langle g,g \rangle_\mu -\langle g, 1\rangle_\mu^2=0
 \]
 Therefore, $g$ has variance zero, so is constant.


 \subsection{Constructing plenty of typical couples : the dimension 3 case} \label{constrplentydim3}

 \subsubsection*{The candidate to be the measure $\eta$, in dimension $3$}

First, recall that $N$ is identified with $\R^{d-1}=\R^2$. Fix also an isomorphism $U\simeq \R$, so that
the set $U^+$ of positive elements is well defined. 

Consider the map $\widetilde{\mathcal{R}}: \widetilde{\Omega}^2\to \widetilde{\Omega}^2$ defined as follows. 
The image $(\x',\y')$ of $(\x,\y)$ is the unique couple such that $x'^+=x^+=y'+$, $x'^-=x^-$, $y'^-=y^-$, 
$t_{\x'}=t_\x=t_\y$, and $\x^+,\y^+$ are the unique frames such that there exists $u\in U^+$ 
with $\x' u=\y'$.

\begin{figure}[ht!]\label{align}
\begin{center}
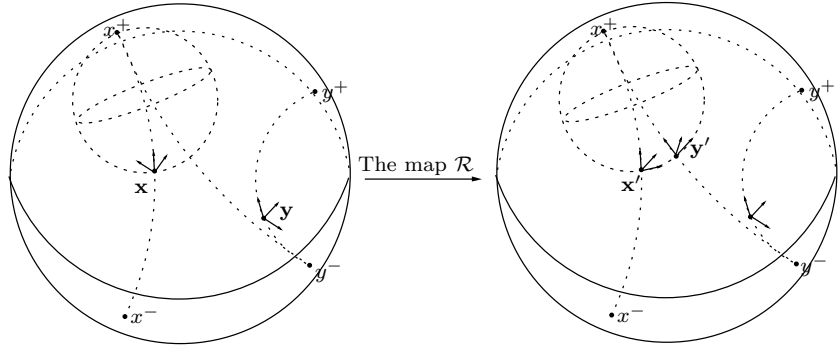
\caption{The alignement map $\mathcal{R}$ } 
\end{center}
\end{figure} 

Consider the restriction of this map to couples $(\x,\y)$ inside some fundamental domain
 for the action of $\Gamma$ on $\widetilde{\Omega}$, so that 
we get a well defined map $\mathcal{R}:\Omega^2\to  \Omega^2$. 
Define $\eta$ as the image $\eta:= \mathcal{R}_*(\mu\otimes\mu)$. 

Observe that condition (1) in \ref{plenty} is automatic, as being typical depends only on $x^-$ and $y^-$. Remark \ref{eta-supported-on-U-orbits}
shows that condition (3) is also automatic. By Remark \ref{pas-besoin-de-L2}, we  only need to show that its projections $ (p_1)_*\eta$ and $ (p_2)_*\eta$ are
absolutely continuous w.r.t. $\mu$. That is the crucial part of the proof. We do it in the next sections. 

The key assumption will of course be our dimension assumption on $\delta_\Gamma>\dim N-\dim U$. Then, we will try to
follow the classical strategy of Marstrand, Falconer, Mattila. However, a new technical difficulty will appear, 
because we will need to do radial projections on circles instead of orthogonal projections on lines. 
The length of the proof below is due to this technical obstacle. 

\subsubsection*{Projections}
First of all, by lemma \ref{energie-finie}, we can  restrict the measure $\nu$ to some subset $A\subset\Lambda_\Gamma$ of measure as close to $1$ as we want, with $I_1(\nu_{|A})<\infty$. 
In the sequel, we denote by $\nu_A$ the measure restricted to $A$ and normalized to be a probability measure. 
 Fix four disjoint compact subsets $X_+,X_-,Y_+,Y_ -$ of $A\subset\Lambda_\Gamma$, 
each of positive $\nu $-measure, and write $\nu_{X_+},\nu_{X_-},\nu_{Y_+},\nu_{Y_-}$ for the Patterson measures restricted to each of these sets, 
normalized to be probability measures. Therefore, all their energies $I_1(\nu_{X_\pm})$ 
and $I_1(\nu_{Y_\pm})$ are finite. 
 
In fact, the definition of the measure $\eta$ will be slightly different than said above. 
First, $\tilde{\eta}$ will be the image by the projection map $\widetilde{\mathcal{R}}$ 
defined above of the restriction of $\tilde{\mu}\otimes\tilde{\mu}$ to the set of couples $(\x,\y)\in\widetilde{\Omega}^2$, 
such that $x^\pm\in X_\pm$ and $y^\pm \in Y_\pm$, $t_\x\in [0,1]$, $t_\y\in[0,1]$.  Then $\eta$ will be defined
on $\Omega^2$ as the image of $\tilde{\eta}$. 
 \medskip

 Pick   two distinct points outside $X_+$, called 'zero' and 'one'. 
For any $x^+\in X_+$, we identify $\partial\H^3\setminus \{x^+\}$ to the complex plane $\C$
by the unique homography, say $h^{x^+}:\partial\H^3\to  \C\cup\{\infty\}$, sending $x^+$ to $+\infty$, zero to $0$ and one to $1$. 
We get a well defined parametrization of angles, as soon as $x^+$ is fixed. 

\begin{rema}\label{varierxplus} Observe that when $x^+$ varies in the compact set $X_+$, as $0$ and $1$ do not belong to $X_+$, 
all the quantities defined geometrically (projections, intersections of circles, ...) vary analytically in $x_+$. 
\end{rema}

In particular, if $\x\in \Omega$ is a frame, 
the frame $\x^+$ in the boundary determines a unique half-circle from $x^+$ to $x^-$ in 
$\partial\H^3$, which is tangent to the first direction of $\x ^+$ at $x^+$, and therefore, a unique half-line originating from $x^-$ in $\C\simeq \partial \H^3\setminus\{x^+\}$. 
We use therefore an angular coordinate $\theta_\x\in [0,2\pi)$ instead of $\x^+$. 

 Let $\vec{u}_\theta$ be the unit vector $e^{i(\theta+\pi/2)}$ in the complex plane. 
Define the projection $ \pi_\theta^{x^+}$ in the direction $\theta$ from
$\partial \H^3\setminus\{x^+\}$ to itself as $\pi_\theta^{x^+}(z)=z.\vec{u}_\theta$. 
Observe that the line   $\R\vec{u}_\theta$ in $\C$, orthogonal to $\theta$, has a canonical parametrization, 
and a Lebesgue measure, denoted by $\ell^{x^+}_\theta$. 

\begin{figure}[ht!]\label{angles}
\begin{center}
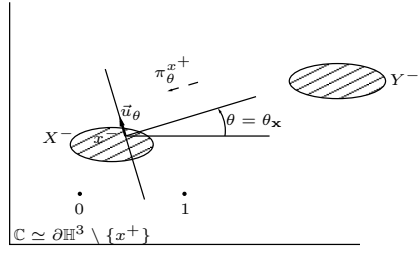
\caption{Angular parameter on $\mathbb{C}\simeq\partial\H^3\setminus\{x^+\}$ } 
\end{center}
\end{figure} 
Once again, the variations of $x^+\mapsto \pi_\theta^{x^+}$ and $x^+\mapsto \ell^{x^+}_\theta$ 
are as regular as possible.
For measures, it means that the Lebesgue measures $\ell^{x^+}_\theta$ are equivalent one to another when $x^+$ varies, with analytic 
Radon-Nikodym derivatives in $x^+$ in restriction to any compact set of $\partial\H^d$ which does not contain
$x^+$.

Observe also that when $x^+$ varies in $X_+$, the distances $d^{x^+}$ induced by the complex metric on $\C\simeq \partial\H^3\setminus\{x^+\}$, 
when restricted to the compact set $X_-\cup Y_-$, are uniformly equivalent to the usual metric on $\partial \H^3$. 
In particular, if we denote by $I_1^{x^+}$ the energy of a measure relatively to the distance
$d^{x^+}$,  there exists a constant $c=c(X_+,X_-,Y_+)$ such that for all $x^+\in X_+$, 
\begin{eqnarray}\label{energies-unif-equiv}
\frac{1}{c}I_1(\nu_A)\le I_1^{x^+}(\nu_{X_-})\le c I_1(\nu_A) \quad\mbox{and}\quad 
\frac{1}{c}I_1(\nu_A)\le I_1^{x^+}(\nu_{Y_-})\le c I_1(\nu_A) 
\end{eqnarray}

 Rephrasing  Marstrand's projection Theorem in dimension 2, we have:
 
\begin{theo}(Falconer, \cite[p82]{MR2118797}, Mattila \cite[th 4.5]{Mattila})
\label{projection-Falconer} 
Assume that $I_1(\nu_A)<\infty$. 
Then for all fixed $x^+\in X^+$, and almost all $\theta\in[0,\pi)$, 
the projection $(\pi_\theta^{x^+})_*\nu_{Y^-}$ 
(resp. $(\pi_\theta^{x^+})_*\nu_{X^-}$) is absolutely continuous
w.r.t $\ell_\theta^{x^+}$. 
Moreover, the map $H^{x^+}$ defined as 
\[
H^{x^+}:(\theta,\xi)\in [0,\pi)\times \R \mapsto \frac{d(\pi_\theta^{x^+})_*\nu_{Y^-}}{d\ell^{x^+}_\theta}(\xi)
\] belongs to $L^2([0,\pi)\times \R)$, and we have
$\|H^{x^+}\|^2_{L^2([0,\pi)\times \R)}\le C I_1(\nu_A)$, 
with $C$ a universal constant which does not depend on $x^+\in X_+$.

In particular, as the variation in $x^+$ is analytic and $X^+$ compact, 
the map $(x^+, \theta,\xi)\to H^{x^+}_\theta(\xi)$ belongs to $L^2(X^+\times [0,\pi]\times\R)$, with $L^2$-norm bounded by the same upper bound $C I_1(\nu_A)$. 

The  same result is true when replacing $Y^-$ with $X^-$. 
\end{theo}
  
\begin{proof}
Thanks to the comparison (\ref{energies-unif-equiv}) between the different notions of energy, we can replace $I_1^{x^+}(\nu_{X_+})$ by $I_1(\nu_A)$, 
and get the desired result.  
\end{proof}


\subsubsection*{Hardy-Littlewood Maximal Inequality}

Let $H^{x^+}_\theta$ be the map 
\[H^{x^+}_\theta:\xi\in\R.\vec{u}_\theta\mapsto \frac{d(\pi_\theta^{x^+})_*\nu_{Y^-}}{d\ell^{x^+}_\theta}(\xi)
\] 
Its maximal function is defined as
\[
MH^{x^+}_\theta(t)=\sup_{\varepsilon>0}\frac{1}{2\varepsilon}\int_{t-\varepsilon}^{t+\varepsilon}\frac{d(\pi_\theta^{x^+})_*\nu_{Y^-}}{d\ell^{x^+}_\theta}(\xi)d\xi=
\sup_{\varepsilon>0}\frac{1}{2\varepsilon}\nu_{Y^-}(\{y\in Y^-, \pi_\theta^{x^+}(y)\in [t-\varepsilon,t+\varepsilon]\})\,.
\]

The strong maximal inequality of Hardy-Littlewood \cite{HL} with $p=2$ on $\R$ (of dimension $1$) asserts that there exists $C=C_{2,1}$ independent of $\theta$ 
such that  for all $\theta\in[0,\pi)$,  
\[
 \|MH^{x^+}_\theta\|_{L^2(\R)}\le C_{2,1}\|H_\theta^{x^+}\|_{L^2(\R)}
\]

We deduce that 
\[
\|MH^{x^+}\|_{L^2([0,\pi)\times\R)} \le \int_0^\pi C_{2,1}^2 \|H^{x^+}_\theta\|_{L^2(\R)}^2d\theta =C_{2,1}^2 \|H^{x^+}_\theta\|_{L^2([0,\pi)\times\R)}^2 <+\infty
\]

The above also holds for the map $G^{x^+}$ defined by \[G^{x^+}_\theta:\xi\in\R.\vec{u}_\theta\mapsto \frac{d(\pi_\theta^{x^+})_*\nu_{X^-}}{d\ell^{x^+}_\theta}(\xi)\,,
\] 
with the same constants. 
 

\subsubsection*{A geometric inequality}

We want to show that the projections $( p_i)_*\eta$ on $\Omega$ are absolutely continuous w.r.t. $\mu$. 
We will first prove it for $ p_1$, and then observe that for $ p_2$, the 
situation is completely symmetric, when reversing the role of $x^-$ and $y^-$. 

Given a Borel set $P=E_+\times E_-\times E_t\times E_\theta\subset X_+\times X_-\times [0,1]\times [0,2\pi)$, observe 
that 
\[
( p_1)_*\eta(P)=\\
\tilde{\mu}\otimes\tilde{\mu}(\{(\x,\y)\in \widetilde{\Omega}^2,\,
x^+\in E_+,x^-\in E_-,  t_\x\in E_t,   y^-\in C^{x^+}(x^-,E_\theta)\,\}
\]
where $C^{x^+}(x^-,E_\theta)$ is the cone of center $x^-$ with angles in $E_\theta$ in the complex plane $\C\simeq\partial\H^3\setminus\{x^+\}$. 

Similarly, 
\[
( p_2)_*\eta(P)=\tilde{\mu}\otimes\tilde{\mu}(\{(\x,\y)\in \widetilde{\Omega}^2,\,
x^+\in E_+,y^-\in E_-,t_\x\in E_t,   x^-\in C^{x^+}(y^-,E_\theta)\,\}\,.
\]

\begin{lem} To prove that $( p_1)_*\eta$ (resp. $( p_2)_*\eta$) is absolutely continuous w.r.t. $\mu$, it is enough to show that there
exists a nonnegative measurable map $F_1$ (resp. $F_2$) such that for all rectangles $P=E_+\times E_-\times E_t\times E_\theta\in X_+\times X_-\times [0,1]\times M$ 
(resp. $P=E_+\times E_-\times E_t\times E_\theta\in X_+\times Y_-\times [0,1]\times M$ )
we have 
\[
( p_1)_*\eta(P)\le \int_P F_1(x^+,x^-,\theta)d\nu_{X^+}(x^+)d\nu_{X^-}(x^-)dt d\theta
\]  and 
\[
( p_2)_*\eta(P)\le \int_P F_2(x^+,y^-,\theta)d\nu_{X^+}(x^+)d\nu_{Y^-}(y^-)dt d\theta
\]
with $F_1\in L^1(\nu_{X^+}\times\nu_{X_-}\times [0,\pi])$, and $F_2\in L^1(\nu_{X^+}\times\nu_{Y_-}\times [0,\pi))$ 
\end{lem}
\begin{proof}  It is clear that $\mu(P)=0$ will imply $( p_1)_*\eta(P)=0$ for all rectangles. 
As they generate the $\sigma$-algebra of $\widetilde{\Omega}\cap(X_+\times X_-\times[0,1]\times[0,\pi)$ 
it implies that $( p_1)_*\eta$ is absolutely continuous w.r.t. $\mu$.
The proof is the same with $ p_2$. 
\end{proof}

Let us show that such integrable maps $F_1$ and $F_2$ exist. 

In fact, we will prove that for all given $x^+$, $F_i(x^+,.)$ is integrable. 
And the fact that, as usual, the variation 
of all involved quantities in $x^+$ is analytic will imply that 
$\|F_i(x^+,.)\|$ is integrable also in $x^+$.

As said above, for $P=E_+\times E_-\times E_t\times E_\theta$  we have 
\[
( p_1)_*\eta(P)=\int_{E_+\times E_-\times E_t}\int_{Y^-} {\bf 1}_{C^{x^+}(x^-,E_\theta)}(y^-)d\nu_{Y_-}(y^-)d\nu_{X_-}(x^-)d\nu_{X_+}(x^+) dt 
\]

Now, we wish to study the quantity $\nu_{Y^-}(C^{x^+}(x^-,E_\theta))$ in order to  
prove that, $x^+$ being fixed, the radial projection of $\nu_{Y^-}$ on the circle of directions around $x^-$ is absolutely continuous 
w.r.t the Lebesgue measure $d\theta$, and control the norm of the Radon-Nikodym derivative, which a priori depends on, and 
needs to be integrable in the variable $x^+$. 

It seems now appropriate to use Theorem \ref{projection-Falconer} to conclude.
 Unfortunately, we have to prove that a radial projection is absolutely continuous, whereas
Theorem \ref{projection-Falconer} deals with orthogonal projection in a certain direction. 
The Hardy-Littlewood maximal $L^2$-inequality will allow us to overcome this difficulty. 

Denote by $\Theta^{x^+}(x^-,y^-)$ the angle in $\partial \H^3\setminus \{x^+\}\simeq \C$ at $x^-$ of the half-line from $x^-$ to $y^-$. 

First, as the distance from $X^-$ to $Y^-$ is uniformly bounded from below, the cone $C^{x^+}(x^-,[\theta_0-\varepsilon,\theta_0+\varepsilon])$
 intersected with $Y^-$ is uniformly 
included in a rectangle of the form $\{y^-\in Y^-, |\pi_{\theta_0}^{x^+}(y^-)-\pi_{\theta_0}^{x^+}(x^-)|\le c_0\varepsilon\}$, for some uniform
constant depending only on the sets $X_\pm$ and $Y_\pm$, and not on $\varepsilon, x^\pm,y^\pm$. In particular, the following result holds. \\

\begin{figure}[ht!]\label{radial-orthogonal} 
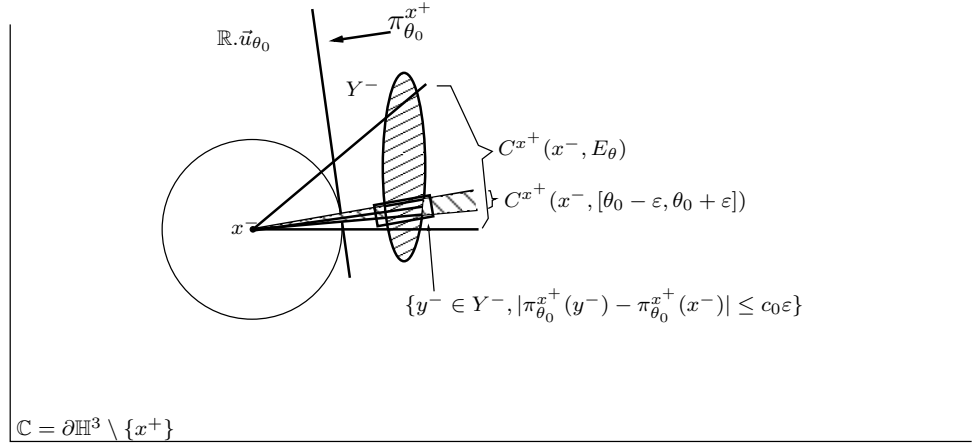
\caption{Radial versus orthogonal projections of $\nu_{Y^-}$ }  
\end{figure} 

\begin{lem}\label{radial-versus-orthogonal} There exists a geometric constant $c_0>0$ depending only on the sizes and respective distances of the sets $X_\pm$ and $Y_\pm$, 
such that 
\[
\nu_{Y_-}(  C^{x^+}(x^-,[\theta_0-\varepsilon,\theta_0+\varepsilon])\cap Y^-)\le 2c_0\varepsilon MH_{\theta_0}^{x^+}(\pi_{\theta_0}^{x^+}(x^-)) 
\]
\end{lem}


\subsubsection*{Conclusion of the argument}

The above inequality does not allow directly to conclude. 
Let us integrate it in $\theta$, to recover the $L^2$-norm of the maximal Hardy-Littlewood function. 
The first inequality follows from the inclusion $[\theta_0-\varepsilon,\theta_0+\varepsilon]\subset [\theta-2\varepsilon,\theta+2\varepsilon]$ for $\theta$ in the first interval, 
the second inequality from Lemma \ref{radial-versus-orthogonal}. 
\begin{eqnarray*}
\nu_{Y_-}(&C^{x^+}&(x^-,[\theta_0-\varepsilon,\theta_0+\varepsilon])\cap Y^-)\\
&\le &\int_{\theta_0-\varepsilon}^{\theta_0+\varepsilon}\nu_{Y^-}(\{y^-\in Y^-, \Theta^{x^+}(x^-,y^-)\in [\theta-2\varepsilon,\theta+2\varepsilon]\})\,\frac{d\theta}{2\varepsilon}\\
&\le& 4c_0\varepsilon\int_{\theta_0-\varepsilon}^{\theta_0+\varepsilon} MH_{\theta}^{x^+}(\pi_{\theta}^{x^+}(x^-))\frac{d\theta}{2\varepsilon} \\
&=& 2c_0 \int_{\theta_0-\varepsilon}^{\theta_0+\varepsilon}MH_{\theta}^{x^+}(\pi_{\theta}^{x^+}(x^-))\,d\theta
\end{eqnarray*}

Define $F_1(x^+,x^-,\theta)$ as 
$$
F_1(x^+,x^-,\theta)=2c_0MH_{\theta}^{x^+}\left(\pi_{\theta}^{x^+}(x^-)\right)=
2c_0\sup_{\epsilon>0}\frac{1}{2\varepsilon}\int_{\pi_\theta^{x^+}(x^-)-\varepsilon}^{\pi_\theta^{x^+}(x^-)+\varepsilon} H_\theta^{x^+}(t)dt\,.
$$
The absolute continuity of $\pi_\theta^{x^+}$ w.r.t $\ell_\theta$, the Cauchy-Schwartz inequality and
the Hardy-Littlewood maximal inequality imply that
\begin{eqnarray*}
\|F_1(x^+,.,.)\|_{L^1([X^-\times[0,\pi])}&=& 2c_0 \int_{X^-}\int_0^\pi MH_{\theta}^{x^+}(\pi_{\theta}^{x^+}(x^-)) d\nu_{X^-}(x^-)d\theta\\
&=&\int_\R\int_0^\pi MH_{\theta}^{x^+}(\xi)\frac{d(\pi_\theta^{x^+})_*\nu_{X^-}}{d\ell_\theta^{x^+}}(\xi)d\xi d\theta\\
&\le & \|MH^{x_+}\|_{L^2(\R\times[0,\pi])}\times\left\|\frac{d(\pi_\theta^{x^+})_*\nu_{X^-}}{d\ell_\theta^{x^+}}(\xi)\right\|_{L^2(\R\times[0,\pi])}\\
&\le& C_1\|H^{x^+}\|_{L^2(\R\times[0,\pi])}\times\|G^{x^+}\|_{L^2(\R\times[0,\pi])}
  \end{eqnarray*} 
 which is, by Projection Theorem \ref{projL2}, 
  bounded from  above by $C_1^2 I_1(\nu_A)^2<\infty$. 

The uniformity of the bound in $x^+\in X^+$ allows to integrate once again the above quantities and deduce that $F_1\in L^1(X^-\times X^+\times[0,\pi])$. 


\subsection{The higher dimensional case}
\label{higherdimeta}

In higher dimension, the strategy of the proof is similar. 
 We want to build a measure $\eta$ on $\Omega^2$ which gives positive measure 
to {\em   plenty of couples on the same $U$-orbit}.

We will build $\eta$ from the measure $\mu\otimes\mu$, to obtain a measure
defined on (a subset of)  $\{(\x,\y)\in\Omega^2,\,\x U=\y U\}$, which gives full measure to typical couples 
$(\x,\y)$ (whose negative orbit satisfies Birkhoff ergodic theorem for the diagonal action of $a_{-1}$, and
whose projections $(p_1)_*\eta$ and $(p_2)*\eta$ on $\Omega$  are absolutely continuous w.r.t $\mu$.

Contrarily to the dimension $3$ case, we will not define any "alignment map". Indeed, 
given a typical couple $(\x,\y)$, one can begin as in dimension $3$, and try to 
find a frame $\x'\in \x M$ and a frame $\y'\in \x' U$ (or in other words $\y'U=\x'U$), so that in particular $y'^+=x'^+=x^+$, 
with the same past as $\y$ (that is, $y'^-=y^-$).
However, there is no canonical choice of such $\x'$, $\y'$, due to the fact that 
the dimension and/or the codimension of $U$  in $N$ will be greater than one. 

Therefore, we will directly define the new measure $\eta$, 
by a kind of averaging procedure of all good choices of
couples $(\x',\y')$.  \\

Identify  the horosphere $\x NM=\x MN$ in $T^1\H^d$
with a $d-1$-dimensional affine space. As in dimension $3$, we wish that the frames $\x'$ and $\y'$ have their
first vectors on $\x NM$, that $\x'$ belongs to the fiber $\x M$ of the vector $\pi_1(\x)$, and $y'^-=y^-$, so that $\y'$ belongs to the fiber $\y' M$ (with an abuse of notation, as $\y'$ is not well defined) of the well defined vector $v_\y=(y^-,x^+,t_\x)$ of $\x MN$. 

These vectors $\x M$ and $\y' M$ are well defined, so that the line from $\x M$ to $\y' M$ in the affine space
$\x NM$ is also well defined. 

Now, given any two frames $\x'$ and $\y'$ in the respective fibers of $\x M$ and $\y' M$, such that $\x' U=\y' U$,
 the $k$-dimensional oriented linear space $P=\x'UM$ contains the line from $\x M$ to $\y' M$. 
The set of such $P$ can be identified with ${ SO}(d-2)/\left(  SO (k-1)\times{ SO}(d-k-1)\right)$. 
 
We will first choose randomly $P$ using the ${ SO}(d-2)$-invariant measure on the latter space. 
Now, given $P$, the set of frames $\x'$ such that the direction of the affine subspace $\x'UM$ is $P$ can be identified with ${  SO}(k)\times {  SO}(d-k-1)$, and we choose $\x'$ randomly using the Haar measure of this group. This determines the element $u \in U$ such that $\x'uM=\y'M$, so it determines $\y'=\x'u$ completely. \\

As in dimension 3, the non-trivial part is to show that the measure obtained by this construction has absolutely continuous marginals. We first describe more precisely the construction to fix notations.


\subsubsection{Restriction of the support of $\mu\otimes\mu$}

Recall that the lift $\tilde{\mu}$ of the measure $\mu$ on $\widetilde{\Omega}$
can be written locally as 
\[
d\tilde{\mu}(\x)=d\nu(x^-)d\nu(x^+)dt_{\x}dm\,,
\]
where $dm$ denotes the Haar measure on the fiber $\x M$ over $\pi_1(\x)$. 
Remember that a frame $\x$ with first vector $\pi_1(\x)$ induces (by parallel transport until 
infinity) a frame at infinity in $T^1_{x^+}\partial \H^d$, or  $T^1_{x^-}\partial \H^d$, so that
$dm$ can also be seen as the Haar measure on the set of frames based at $x^-$ inside $T^1_{x^-}\partial\H^d$.  

As in dimension $3$, consider a subset $A\subset\Lambda_\Gamma$ of positive $\nu$-measure such that $I_1(\nu_{|A})<\infty$.
Choose four compact sets   $X^\pm, Y^\pm$ inside $A$, pairwise disjoint, 
and restrict $\tilde{\mu}\otimes\tilde{\mu}$
to the couples $(\x,\y)\in\widetilde{\Omega}^2$ such that $x^\pm\in X^\pm$ and $y^\pm \in Y^\pm$, and $t_\x,t_\y\in [0,1]$.   

\subsubsection{Coordinates on $\partial\mathbb{H}^d$}

 For the purpose of contructing $\eta$, it will be convenient to have a family of identifications of horospheres, or here the complement of a point $x^+$ in $\partial \H^d$, with  the vector space $\R^{d-1}$.
Let $(e_i)_{1\le i\le d-1}$ be the canonical basis of $\R^{d-1}$. 
Choose three different points $x_0^+ \in X^+$, $x_0^-\in X^-$ and $y_0^⁻\in Y^-$, in the support of $\nu_{|X^+},\nu_{|X^-}$ and $\nu_{|Y^-}$ respectively.

Now we want to get a unique homography $h_{x^+}$ from $\partial\H^d\setminus\{x^+\}$ 
to $\R^{d-1}$ sending $x_0$ to $0$, $y_0$ to $e_1$, 
and $x^+$ to infinity, with a smooth dependence in $x^+$. 

To do so, choose  successively $d-3$ other points, say $q_2$, ... $q_{d-2}$ in $\partial\H^d$, in such a way that, uniformly in $x^+\in X^+$, 
none of the points $x^+,x_0, y_0, q_2,\dots, q_{d-2}$ belongs to a circle containing three other points.
Now, it is elementary to check that there is a unique conformal map $h_{x^+}$ sending
$x^+$ to infinity, $x_0$ to $0$, $y_0$ to $e_1$, $q_2$ inside the half-plane $\R.e_1 + \R_+ e_2$, $q_3$ 
inside the half-space $\R e_1+\R e_2+\R_+ e_3$, and so on up to $q_{d-2}$. 
This is the desired map. 

Up to decreasing the size of $X^+$, $X^-$ and $Y^-$ using neighbourhoods of $x_0^+,x_0^-,y_0^-$ respectively,
we can moreover assume that for all these conformal maps uniformly in $x^+\in X^+, x^-\in X^-, y^-\in Y^-$, 
the first coordinate of the vector $\overrightarrow{h_{x^+}(x^-)h_{x^+}(y^-)}$ belongs to $[\frac{1}{2}, 2]$, and the norm of this vector is bounded by $3$.
In the sequel, we use the coordinates induced by $h_{x^+}$ on $\partial\H^d$. 

\subsubsection{A nice bundle}

We will construct a measure $\tilde{\eta}$ on the set 
$$
\mathcal{S}_\eta=\{(\x,\y)\in\Omega^2,\, : \, x^+=y^+ \in X^+, x^-\in X^-, y^-\in Y^- , \x U=\y U\}\,,
$$ 
and prove that it satisfies assumptions (1),(2),(3) of Lemma \ref{reductiontwo}, so that Theorem \ref{ergodicity} follows. 
Observe that this space $\mathcal{S}_\eta$ is a fiber bundle over some subset
$$
 \mathcal{P} \subset X^+\times X^-\times Y^-\times \mathcal{G}_{k }^{d-1}\, , 
$$
 whose projection is simply
$$
(\x,\y)\in \mathcal{S}_\eta\to (x^+,x^-,y^-,Vect(x_1,\dots, x_{k})) \in\mathcal{P} \,,
$$
where $Vect(x_1,\dots,x_{k})$ is the oriented $k$-linear space spanned by the $k$ first vectors of
the frame $\x^+$ at infinity with orientation $x_1\wedge...\wedge x_{k}$, 
or equivalently the $k$-plane spanned by these $k$ vectors viewed around $x^-$ at
infinity, i.e. inside $\R^{d-1}$ identified with $\partial\H^d\setminus\{x^+\}$ using the map $h_{x^+}$.

Moreover, observe that it is a principal bundle, whose fibers
are isomorphic  to ${SO}(k)\times { SO}(d-1-k)\times A$. 
Indeed, given a couple $(\x,\y)$ in the fiber of $(x^+,x^-,y^-,P)$, after maybe let $A$ act diagonally
so that both couples are based on the horosphere passing through the origin $o\in \H^d$, any other 
couple differs from $(\x,\y)$ only by changing $(x_1,\dots, x_{k})$ 
into another orthonormal basis of 
$Vect(x_1,\dots, x_{k})$, and $(x_{k+1},\dots,x_{d-1})$ into another orthonormal basis of 
$Vect(x_{k+1},\dots, x_{d-1})$, preserving the orientation.


\subsubsection{Defining the measure}

Given $x^+\in X^+$, we first define a measure $\bar{\eta}_{x^+}$ supported 
on the set 
\[
\mathcal{P}_{x^+}=\{(x^-,y^-,P) \; : \;  x^-\in X^-, y^- \in Y^-, P\in \mathcal{G}_{k}^{d-1}\,, s.t.   \,\,
\overrightarrow{h_{x^+}(x^-)h_{x^+}(y^-)} \in P  \} \,.
\] 
(a subset of $X^-\times Y^-\times \mathcal{G}_{k}^{d-1}$) as follows.

Observe that, thanks to our choice of coordinates, 
the vector $\overrightarrow{h_{x^+}(x^-)h_{x^+}(y^-)}$ has always a nonzero coordinate along $e_1$. 
Therefore, any $k$-plane $P$ containing $\overrightarrow{h_{x^+}(x^-)h_{x^+}(y^-)}$ is uniquely determined by its 
$k-1$-dimensional intersection $P\cap e_1^\perp$ with $e_1^\perp$.

Thus, we have a well defined measure on $\mathcal{P}_{x^+}$:
\[
d\bar{\eta}_{x^+}(x^-,y^-,P)=d\nu_{X^-}(x^-) d\nu_{Y^-}(y^-)d\sigma_{k-1}^{d-2}(P\cap e_1^\perp)\,,
\]
where $\sigma_{k-1}^{d-2}$ is the $SO(d-2)$-invariant 
probability measure on the Grassmannian manifold
 of $(k-1)$-planes in $e_1^\perp$.\\

 Now, $\mathcal{P}$ is a bundle over $X^+$ with fibers $\mathcal{P}_{x^+}$. Define $\bar{\eta}$ on $\mathcal{P}$
 as the measure which disintegrates as $\nu_{X^+}$ on the basis $X^+$ and $\bar{\eta}_{x^+}$ in the fibers.

Pick $\epsilon$ small enough, and lift $\bar{\eta}$ to $\widetilde{\eta}$ on $\widetilde{\Omega}^2$, or more precisely 
on its subset 
$$
\widetilde{S}_\eta=\{(\x,\y)\in \widetilde{\Omega}^2, \,\x U=\y U,\,\,t_\x=t_\y\in[0,\epsilon]\} \,
$$ by  endowing the fibers with the Haar measure of ${ SO}(k )\times { SO}(d-1-k)$ times the uniform probability measure on the interval $[0,\epsilon]$.

If $X^\pm,Y^\pm$ and $\epsilon$ are small enough, we can assume that the support of $\widetilde{\eta}$ is 
included inside the product of two single fundamental domains of the action of $\Gamma$ on ${ SO}^o(d,1)$, so that
it induces a well defined measure $\eta$ on the quotient. 

By construction, it is supported on couples $(\x,\y)$ in the same $U$-orbit, and 
as in dimension $3$, it gives full measure to couples $(\x,\y)$ which are typical 
in the past, because this property of being typical depends only on $x^-,y^-$, 
and $\nu_{|X^-}\otimes\nu_{|Y^-}$ gives full measure to the pairs $(x^-,y^-)$ which 
are negative endpoints of typical couples $(\x,\y)$.

The main point to check is that $(p_1)_*\eta$ and $(p_2)_*\eta$ are
absolutely continuous w.r.t. $\mu$.


\subsubsection{Absolute continuity}

Let us reduce the abolute continuity of $(p_i)_*\eta$ to another absolute continuity property, by a succession of
elementary observations. 

First, to prove that $(p_1)_*\eta$ and $(p_2)_*\eta$ are absolutely continuous 
w.r.t. $\mu$, it is sufficient to prove that $(\tilde{p}_1)_*\widetilde{\eta}$ and 
$(\tilde{p}_2)_*\widetilde{\eta}$, where $\tilde{p}_i:\widetilde{\Omega}^2\to \widetilde{\Omega}$ are the coordinates maps, are both absolutely continuous with respect to $\tilde{\mu}$.  
 
Both measures are defined on the compact set
 $$
T=\{ \x \in \widetilde{\Omega} \, : \, t_\x \in [0,\epsilon], \, x^+\in X^+, \, x^- \in (X^-\cup Y^-)\} \, .
$$
 This set $T$ is fibered over  
 $$
X^+\times(X^-\cup Y^-) \times \mathcal{G}_{k}^{d-1} \, ,
$$
 with projection map $\x\to (x^+,x^-,\x MU)$ and fiber isomorphic to $  {  SO}(k)\times {SO}(d-k-1)\times A$.

$$\xymatrix{
{ \widetilde{S}_\eta \subset\widetilde{\Omega}^2  } \ar[d]    \ar[r]^{\tilde{p}_i} & { T\subset\widetilde{\Omega} } \ar[d] \\
{\mathcal{P} \subset X^+\times X^-\times Y^-\times\mathcal{G}_k^{d-1} } \ar[d] 
   \ar[r]^-{\bar{p}_i}   &  {  X^+\times( X^-\cup Y^-)\times \mathcal{G}_k^{d-1} } \ar[d]\\
{X^+}   &  X^+ \\}
$$

On the upper left part of this diagram, observe that the measure $\tilde{\eta}$ disintegrates over $\mathcal{P}$, with the Haar measure of $SO(k)\times SO(d-1-k)\times A$ in the fibers, and $\bar{\eta}$ on $\mathcal{P}$.

Similarly, on the upper right of the diagram, the measure $\tilde{\mu}$ restricted to $T$ disintegrates over $X^+\times (X^-\cup Y^-)\times \mathcal{G}_k^{d-1}$, with measure $\nu_{X^+}\otimes \nu_{X^-\cup Y^-}\times \sigma_k^{d-1}$ on the basis, and Haar measure of $SO(k)\times SO(d-1-k)\times A$ in the fibers. 

Therefore, to prove that $(\tilde{p}_i)_*\tilde{\eta}$ is absolutely continuous w.r.t.
 $\tilde{\mu}$, it is enough to 
prove that $(\bar{p}_i)_*\bar{\eta}$ is absolutely continuous w.r.t. $\nu_{X^+}\otimes \nu_{X^-\cup Y^-}\times \sigma_k^{d-1}$.

Look at the lower part of the diagramm now. 
 The measure $\bar{\eta}$ itself  disintegrates over $X^+$, 
with $\nu_{X^+}$ on the base and $\bar{\eta}_{x^+}$ on each fiber $\mathcal{P}_{x^+}$, 
whereas the measure $(\bar{p}_i)_* \bar{\eta}$ disintegrates
also over $\nu_{X^+}$, with measure $\nu_{X^-\cup Y^-}\times \sigma_k^{d-1}$ on each fiber. 

Thus, it is in fact enough to prove that for $\nu_{X^+}$-almost every $x^+$, the image of the 
measure $\bar{\eta}_{x^+}$ under the natural projection map $\mathcal{P}_{x^+}\to \{x^+\}\times X^-\cup Y^-\times \mathcal{G}_k^{d-1}$ is absolutely continuous w.r.t. $\nu_{X^-\cup Y^-}\otimes \sigma_k^{d-1}$. 

The precise statement that we will prove is Lemma \ref{abscontdimsup}. By the above discussion, it implies that $(p_i)_*\eta$
is absolutely continuous w.r.t. $\mu$, and therefore, as in dimension $3$, Theorem \ref{ergodicity} follows from Lemma \ref{reductiontwo}.


\subsubsection{Absolute continuity of conditional measures}

 We  discuss now the absolute continuity of the marginals laws of $\bar{\eta}_{x^+}$. 

In order to do so, it is necessary to say a few words about the distance on the  Grassmannian manifolds of oriented subspaces that we shall use. As we are only interested in the local properties of the distance, we will (abusively) define it only on the Grassmannian manifold of unoriented subspaces. 

If $P$ is a $l$-dimensional subspace of a Euclidean space of dimension $n$, we write $\Pi_P$ for the orthogonal projection on $P$. 
If $P,P' \in \mathcal{G}_l^n$ are two $l$-dimensional subspaces, a distance between $P$ and $P'$ can be
 defined as the operator norm of $\Pi_P-\Pi_{P'}$ (which is also the operator norm of
$\Pi_{P^\perp}-\Pi_{(P')^\perp}$). 

We will use the following facts.
 \begin{enumerate} 
  \item The above distance is Lipschitz-equivalent to any Riemannian metric on 
$\mathcal{G}_l^n$, and $\sigma_l^n$ is a smooth measure. In particular, up to multiplicative constants, the measure of a ball of sufficiently small radius $r$ around a point $P$ is 
 $$
\sigma_{l}^{n}(B_{\mathcal{G}_{l}^{n}}(P,r)) \simeq r^{l(n-l)}.
$$
  \item Identify $e_1^\perp$ with $\R^{d-2}$. Define
  $$
(\mathcal{G}_{k}^{d-1})'=\{P \in \mathcal{G}_{k}^{d-1} \; : \; P \not\subset e_1^\perp\}.
$$  The map $P \in (\mathcal{G}_k^{d-1})'\mapsto P\cap e_1^\perp \in \mathcal{G}_{k-1}^{d-2}$ is well-defined and smooth, so that its restriction to any compact set is Lipschitz. 
  
  \item Let $P, P_1$ be two $k$-dimensional subspaces of $\R^{d-1}$. If $v\in P$, $\|v\|\leq 3$ and $d_{\mathcal{G}_k^{d-1}}(P,P_1)\leq r$,
  then 
  $$
  \| \Pi_{P_1^\perp}(v) \| \leq 3r .
  $$
 \end{enumerate}

 \begin{lem} \label{abscontdimsup}
 There exist two functions $F_{x^+,1} \in L^1(\nu_{X^-}\otimes \sigma_k^{d-1})$, $F_{x^+,2} \in L^1(\nu_{Y^-}\otimes \sigma_k^{d-1})$ such that for any $E  \subset (X^-\cup Y^-)$, any ball 
$B=B(P_0,r)\subset \mathcal{G}_k^{d-1}$  of sufficiently small radius $r$ around some $P_0 \in \mathcal{G}_k^{d-1}$, and any $x^+ \in X^+$,
  $$
\bar{\eta}_{x^+}(\{ (x^-,y^-,P)\in \mathcal{P}_{x^+} \; : \; (x^-,P) \in E\times B \})
  \leq \int_{E\times B} F_{x^+,1} \, d\nu_{X^-}\otimes \sigma_k^{d-1},
$$
  and 
  $$
\bar{\eta}_{x^+}(\{ (x^-,y^-,P)\in \mathcal{P}_{x^+} \; : \; (y^-,P) \in E\times B \})
  \leq \int_{E\times B} F_{x^+,2} \, d\nu_{ Y^-}\otimes \sigma_k^{d-1}.
$$
  Moreover, the $L^1$-norms of $F_{x^+,i}$ are uniformly bounded on $X^+$.
 \end{lem}
 \begin{proof} We prove only the second inequality, the first one is similar and only exchanges the roles of $x^-$ and $y^-$ in the following. \\
  First choose some $P_1 \in B_{\mathcal{G}_k^{d-1}}(P_0,r)$.
  If $(x^-,y^-,P)\in \mathcal{P}_{x^+}$ with  $P \in B_{\mathcal{G}_k^{d-1}}(P_1,2r)$, then, provided $r$ is small enough, both $P$ and $P_1$ are in a fixed compact subset of  $(\mathcal{G}_{k}^d)'$. This implies that for some fixed $c_0>0$,
  $$ Q=P\cap e_1^\perp \in B_{\mathcal{G}_{k-1}^{d-2}}(P_1\cap e_1^\perp,c_0 r).$$
  We also have 
  $$
d_{P_1^\perp}(\Pi_{P_1^\perp}(x^-),\Pi_{P_1^\perp}(y^-))\leq 6r.
$$
  Thus we have the inequalities
  \begin{align*}
   & \bar{\eta}_{x^+}  (\{ (x^-,y^-,P)\in \mathcal{P}_{x^+} \; : \; (y^-,P) \in E\times B_{\mathcal{G}_k^{d-1}}(P_1,2r) \})\\
  & = \int 1_{B_{\mathcal{G}_k^{d-1}}(P_1,2r)}(Q\oplus \overrightarrow{h_{x^+}(x^-)h_{x^+}(y^-)}) \, d\nu_{X^-}(x^-)  d\nu_{Y^-}(y^-) d\sigma_{k-1}^{d-2}(Q),\\
  & \leq \sigma_{k-1}^{d-2}(B_{\mathcal{G}_{k-1}^{d-2}}(P_1\cap e_1^\perp,c_0 r)) \int_{E}\int_{X^-} 1_{B(\Pi_{P_1^\perp}(y),6r)}(\Pi_{P_1^\perp}(x))d\nu_{ X^-}(x^-)d\nu_{ Y^-}(y^-),\\
  & \leq \sigma_{k-1}^{d-2}(B_{\mathcal{G}_{k-1}^{d-2}}(P_1\cap e_1^\perp,c_0 r)) \int_{E} (\Pi_{P_1^\perp})_*\nu_{ X^-}(B(\Pi_{P_1^\perp}(x),6r))d\nu_{ Y^-}(y^-)\\
  & \leq \sigma_{k-1}^{d-2}(B_{\mathcal{G}_{k-1}^{d-2}}(P_1\cap e_1^\perp,c_0 r)) \int_{E} (6r)^{d-k-1} \, MH_{x^+,P_1}(\Pi_{P_1^\perp}(y))d\nu_{ Y^-}(y^-),
    \end{align*}
 where $MH_{x^+,P_1}$ is the maximal function 
 $$
MH_{x^+,P_1}(v)=\sup_{\rho>0} \rho^{-(d-k-1)}
 \int_{B_{P_1^\perp}(v,\rho)} \frac{d(\Pi_{P_1^\perp}\circ h_{x^+})_* \nu_{ X^-}}{dw}(w) dw\,.
$$

  We now integrate this inequality over $P_1 \in B_{\mathcal{G}_k^{d-1}}(P_0,r)$ 
using the uniform measure and the fact that
$$
B_{\mathcal{G}_k^{d-1}}(P_0,r)\subset B_{\mathcal{G}_k^{d-1}}(P_1,2r).
$$
   We obtain
  \begin{align*}
     & \bar{\eta}_{x^+}  (\{ (x^-,y^-,P)\in \mathcal{P}_{x^+} \; : \; (y^-,P) \in E\times B_{\mathcal{G}_k^{d-1}}(P_0,r) \})\\
     & \leq \int_{B_{\mathcal{G}_{k}^{d-1}}(P_0,r))}
     \bar{\eta}_{x^+}  (\{ (x^-,y^-,P)\in \mathcal{P}_{x^+} \; : \; (y^-,P) \in E\times B_{\mathcal{G}_k^{d-1}}(P_1,2 r) \} 
     \frac{d\sigma_k^{d-1}(P_1)}{\sigma_{k}^{d-1}(B_{\mathcal{G}_{k}^{d-1}}(P_0,r))}\\
     & \leq \int_{E\times B_{\mathcal{G}_{k}^{d-1}}(P_0,r))}
     \frac{ \sigma_{k-1}^{d-2}(B_{\mathcal{G}_{k-1}^{d-2}}(P_1\cap e_1^\perp,c_0 r)) (6r)^{d-k-1}}
     {\sigma_{k}^{d-1}(B_{\mathcal{G}_{k}^{d-1}}(P_0,r)) }
      MH_{x^+,P_1}(\Pi_{P_1^\perp}(y^-))d\nu_{|Y^-}(y^-)d\sigma_k^{d-1}(P_1).
  \end{align*}
Now, the ratio
$$
\frac{ \sigma_{k-1}^{d-2}(B_{\mathcal{G}_{k-1}^{d-2}}(P_1\cap e_1^\perp,c_0 r)) (6r)^{d-k-1}}
     {\sigma_{k}^{d-1}(B_{\mathcal{G}_{k}^{d-1}}(P_0,r)) },
$$
is bounded by a uniform constant $c>0$, 
since the dimension of the Grassmannian manifolds $\mathcal{G}_{r}^{n}$ is $r(n-r)$, 
so the above ratio is comparable, up to multiplicative constants, 
with $\frac{r^{(k-1)(d-k-1)}\times r^{d-k-1}}{r^{k(d-k-1)}}= 1$.\\
This proves an inequality of the desired form with the function
     $$F_{x^+,2}(y^-,P)=c \, MH_{x^+,P}(\Pi_{P^\perp}(h_{x^+}(y^-))).$$
     We still have to show that this function is in $L^1(\nu_{Y^-}\otimes \sigma_k^{d-1})$. 
Let us compute its norm
     \begin{align*}
     \mathcal{N}=&\int_{Y^-\times \mathcal{G}_{k-1}^{d-2}} MH_{x^+,P}(\Pi_{P^\perp}(h_{x^+}(y^-)))\,d\nu_{ Y^-}(y^-)d\sigma_{k}^{d-1}(P)\\
     &=\int_{\mathcal{G}_{k-1}^{d-2}} \left( \int_{P^\perp} MH_{x^+,P}(v)\,d(\Pi_{P^\perp}\circ h_{x^+})_*\nu_{|Y^-}(v)\right) d\sigma_{k}^{d-1}(P)\\
     &=\int_{\mathcal{G}_{k-1}^{d-2}} \left( \int_{P^\perp} MH_{x^+,P}(v)\,\frac{d(\Pi_{P^\perp}\circ h_{x^+})_*\nu_{|Y^-}}{dv}(v) dv\right) d\sigma_{k}^{d-1}(P).
     \end{align*}
     By \cite[Theorem 9.7]{MR1333890}, the two Radon-Nikodym derivatives
     $$\frac{d(\Pi_{P^\perp}\circ h_{x^+})_*\nu_{|Y^-}}{dv}, \; \frac{d(\Pi_{P^\perp}\circ h_{x^+})_*\nu_{|X^-}}{dv},$$
     have the square of their $L^2$-norms bounded by a constant times  the respective energies
     $$I_{d-1-k}((h_{x^+})_* \nu_{|Y^-}), \; I_{d-1-k}((h_{x^+})_* \nu_{|X^-}).$$
     By the Hardy-Littlewood inequality \cite[Theorem 2.19]{MR1333890}, this is also true for their maximal functions, with a different constant. 
     By the choices of $X^+,X^-,Y^-$ and $h_{x^+}$, the family of maps $(h_{x^+})_{x^+\in X^+}$ is uniformly bilispchitz when restricted to the compact set $X^-\cup Y^-$. In particular, the above energies are in turn bounded by a constant times $I_{d-1-k}(\nu)$.\\
     The integral $\mathcal{N}$ is thus the scalar product of two $L^2$ functions, each one of norm less that a fixed multiple of $\sqrt{I_{d-1-k}(\nu)}$.\\
     This implies that there exists a constant $c>0$ such that
     $$\mathcal{N} \leq c \, I_{d-1-k}(\nu).$$
 \end{proof}





\section{Acknowledgments}

The authors thank warmly S\'ebastien Gou\"ezel for all the interesting discussions and useful comments on the subject.



\bibliography{biblio}{}

\begin{thebibliography}{10}

\bibitem{Aar}
Jon Aaronson.
\newblock {\em An introduction to infinite ergodic theory}, volume~50 of {\em
  Mathematical Surveys and Monographs}.
\newblock American Mathematical Society, Providence, RI, 1997.

\bibitem{Ancona}
{A}lano {A}ncona.
\newblock {E}xemples de surfaces hyperboliques de type divergent, de mesure de
  {Su}llivan associ\'ees finies mais non g\'eom\'etriquement finies.
\newblock unpublished.

\bibitem{MR3012160}
Hassan Azad and Indranil Biswas.
\newblock On the conjugacy of maximal unipotent subgroups of real semisimple
  {L}ie groups.
\newblock {\em J. Lie Theory}, 22(3):883--886, 2012.

\bibitem{MR0451307}
Rufus Bowen and Brian Marcus.
\newblock Unique ergodicity for horocycle foliations.
\newblock {\em Israel J. Math.}, 26(1):43--67, 1977.

\bibitem{MR1779902}
Fran{\c{c}}oise Dal'bo.
\newblock Topologie du feuilletage fortement stable.
\newblock {\em Ann. Inst. Fourier (Grenoble)}, 50(3):981--993, 2000.

\bibitem{MR1973093}
N.~De{\u{g}}irmenci and {\c{S}}.~Ko{\c{c}}ak.
\newblock Existence of a dense orbit and topological transitivity: when are
  they equivalent?
\newblock {\em Acta Math. Hungar.}, 99(3):185--187, 2003.

\bibitem{Dufloux2016}
Laurent Dufloux.
\newblock Projections of patterson-sullivan measures and the mohammadi-oh
  dichotomy.
\newblock preprint arxiv 1605.02100, 2016.

\bibitem{Dufloux2017}
Laurent Dufloux.
\newblock The case of equality in the dichotomy of mohammadi-oh.
\newblock preprint arxiv 1701.04555, 2017.

\bibitem{MR0310926}
Patrick Eberlein.
\newblock Geodesic flows on negatively curved manifolds. {I}.
\newblock {\em Ann. of Math. (2)}, 95:492--510, 1972.

\bibitem{MR2118797}
Kenneth Falconer.
\newblock {\em Fractal geometry}.
\newblock John Wiley \& Sons, Inc., Hoboken, NJ, second edition, 2003.
\newblock Mathematical foundations and applications.

\bibitem{MR1934285}
Damien Ferte.
\newblock Flot horosph\'erique des rep\`eres sur les vari\'et\'es hyperboliques
  de dimension 3 et spectre des groupes kleiniens.
\newblock {\em Bull. Braz. Math. Soc. (N.S.)}, 33(1):99--123, 2002.

\bibitem{FS}
Livio Flaminio and Ralf Spatzier.
\newblock Geometrically finite groups, patterson-sullivan measures and ratner's
  rigidity theorem.
\newblock {\em Invent. Math.}, 99:601--626, 1990.

\bibitem{MR2339285}
Yves Guivarc'h and Albert Raugi.
\newblock Actions of large semigroups and random walks on isometric extensions
  of boundaries.
\newblock {\em Ann. Sci. \'Ecole Norm. Sup. (4)}, 40(2):209--249, 2007.

\bibitem{HL}
G.~H. Hardy and J.~E. Littlewood.
\newblock A maximal theorem with function-theoretic applications.
\newblock {\em Acta Math.}, 54(1):81--116, 1930.

\bibitem{MR2732978}
Sa'ar Hersonsky and Fr\'ed\'eric Paulin.
\newblock On the almost sure spiraling of geodesics in negatively curved
  manifolds.
\newblock {\em J. Differential Geom.}, 85(2):271--314, 2010.

\bibitem{Ledrappier-Platon}
Fran\c{c}ois Ledrappier.
\newblock Entropie et principe variationnel pour le flot g\'eod\'esique en
  courbure n\'egative pinc\'ee.
\newblock In {\em G\'eom\'etrie ergodique}, volume~43 of {\em Monogr. Enseign.
  Math.}, pages 117--144. Enseignement Math., Geneva, 2013.

\bibitem{LL}
Fran\c{c}ois Ledrappier and Elon Lindenstrauss.
\newblock On the projections of measures invariant under the geodesic flow.
\newblock {\em Int. Math. Res. Not.}, (9):511--526, 2003.

\bibitem{MR0063439}
J.~M. Marstrand.
\newblock Some fundamental geometrical properties of plane sets of fractional
  dimensions.
\newblock {\em Proc. London Math. Soc. (3)}, 4:257--302, 1954.

\bibitem{MR1333890}
Pertti Mattila.
\newblock {\em Geometry of sets and measures in {E}uclidean spaces}, volume~44
  of {\em Cambridge Studies in Advanced Mathematics}.
\newblock Cambridge University Press, Cambridge, 1995.
\newblock Fractals and rectifiability.

\bibitem{Mattila}
Pertti Mattila.
\newblock Hausdorff dimension, projections, and the {F}ourier transform.
\newblock {\em Publ. Mat.}, 48(1):3--48, 2004.

\bibitem{McMullen}
Curtis McMullen, Amir Mohammadi, and Hee Oh.
\newblock Geodesic planes in hyperbolic 3-manifolds.
\newblock preprint, 2015.

\bibitem{McMullen2}
Curtis McMullen, Amir Mohammadi, and Hee Oh.
\newblock Horocycles in hyperbolic 3-manifolds.
\newblock preprint, 2015.

\bibitem{MO}
Amir Mohammadi and Hee Oh.
\newblock Ergodicity of unipotent flows and {K}leinian groups.
\newblock {\em J. Amer. Math. Soc.}, 28(2):531--577, 2015.

\bibitem{montgomery-samelson}
Deane Montgomery and Hans Samelson.
\newblock Transformation groups of spheres.
\newblock {\em Ann. of Math. (2)}, 44:454--470, 1943.

\bibitem{Peigne}
Marc Peign\'e.
\newblock On the {P}atterson-{S}ullivan measure of some discrete group of
  isometries.
\newblock {\em Israel J. Math.}, 133:77--88, 2003.

\bibitem{Ro}
Thomas Roblin.
\newblock Ergodicit\'e et \'equidistribution en courbure n\'egative.
\newblock {\em M\'em. Soc. Math. Fr. (N.S.)}, (95):vi+96, 2003.

\bibitem{Schapira2015}
Barbara Schapira.
\newblock A short proof of unique ergodicity of horospherical foliations on
  infinite volume hyperbolic manifolds.
\newblock {\em Confluentes Math.}, 8(1):165--174, 2016.

\bibitem{MR1327939}
B.~Stratmann and S.~L. Velani.
\newblock The {P}atterson measure for geometrically finite groups with
  parabolic elements, new and old.
\newblock {\em Proc. London Math. Soc. (3)}, 71(1):197--220, 1995.

\bibitem{Sullivan1979}
Dennis Sullivan.
\newblock The density at infinity of a discrete group of hyperbolic motions.
\newblock {\em Inst. Hautes \'Etudes Sci. Publ. Math.}, (50):171--202, 1979.

\bibitem{MR688349}
Dennis Sullivan.
\newblock Disjoint spheres, approximation by imaginary quadratic numbers, and
  the logarithm law for geodesics.
\newblock {\em Acta Math.}, 149(3-4):215--237, 1982.

\bibitem{Sullivan84}
Dennis Sullivan.
\newblock Entropy, {H}ausdorff measures old and new, and limit sets of
  geometrically finite {K}leinian groups.
\newblock {\em Acta Math.}, 153(3-4):259--277, 1984.

\bibitem{Winter}
Dale Winter.
\newblock Mixing of frame flow for rank one locally symmetric spaces and
  measure classification.
\newblock {\em Israel J. Math.}, 210(1):467--507, 2015.

\end{thebibliography}
\bibliographystyle{plain}

\end{document}